\documentclass[12pt]{amsart}

\usepackage[letterpaper, margin=1.25in]{geometry}
\usepackage[colorlinks=true,citecolor=cyan!45!blue,linkcolor=blue!60!gray]{hyperref}%
\usepackage[english]{babel}
\usepackage[utf8]{inputenc}
\usepackage{subcaption}
\usepackage{amsmath}
\usepackage{amssymb}
\usepackage{amsfonts}
\usepackage{amsthm}
\usepackage{mathtools}
\usepackage{mathrsfs}
\usepackage[all]{xy}
\usepackage{graphicx}
\usepackage{color}
\usepackage{url}
\usepackage{caption}

\usepackage[colorinlistoftodos]{todonotes}
\usepackage{tikz}
\usetikzlibrary{patterns}
\usepackage{young}

\usetikzlibrary{calc}
\usepackage{multicol}

\usepackage{textcomp}

\usepackage[
backend=biber,
style=alphabetic,
sorting=ynt,
natbib=true
]{biblatex}
\theoremstyle{plain}
\newtheorem{thm}{Theorem}
\newtheorem{lemma}[thm]{Lemma}
\newtheorem{cor}[thm]{Corollary}

\newtheorem{prop}[thm]{Proposition}

\theoremstyle{definition}
\newtheorem{defn}[thm]{Definition}

\theoremstyle{remark}
\newtheorem{rem}[thm]{Remark}

\numberwithin{equation}{section}
\numberwithin{thm}{section}

\newcounter{casenum}
\newenvironment{caseof}{\setcounter{casenum}{1}}{\vskip.5\baselineskip}
\newcommand{\case}[2]{\vskip.5\baselineskip\par\noindent {\bfseries Case \arabic{casenum}:} #1\\#2\addtocounter{casenum}{1}}

\newcommand{\oM}{\hspace*{0pt} \mkern 4mu\overline{\mkern-4mu M\mkern-1mu}\mkern 1mu}

\newcommand{\bbz}{\mathbb{Z}}
\newcommand{\bbr}{\mathbb{R}}
\newcommand{\bbc}{\mathbb{C}}

\newcommand{\bbp}{\mathbb{P}}

\newcommand{\bbn}{\mathbb{N}}
\newcommand{\bbt}{\mathbb{T}}

\newcommand{\calo}{\mathcal{O}}
\newcommand{\caln}{\mathcal{N}}
\newcommand{\calc}{\mathcal{C}}
\newcommand{\cald}{\mathcal{D}}

\newcommand{\calq}{\mathcal{Q}}

\newcommand{\cali}{\mathcal{I}}
\newcommand{\calg}{\mathcal{G}}

\newcommand{\eb}[2]{B^c_{#1}(0)\times {#2}/\langle \sigma\rangle}
\newcommand{\ALGb}{\mathrm{ALG}_{\beta}}
\newcommand{\ALGone}{\mathrm{ALG}_1}
\newcommand{\ALGog}{\mathrm{ALG}_{1,\gamma}}
\newcommand{\ALGot}{\mathrm{ALG}_{1,2}}

\newcommand {\metzero}[1]{||{#1}||_{g_0}}
\newcommand {\metinf}[1]{||{#1}||_{g_{\infty}}}

\newcommand {\estim}[1]{O\left(#1\right)}
\newcommand {\estp}[1]{O'\left(#1\right)}

\newcommand \alp{\alpha}
\newcommand \eps{\varepsilon}
\newcommand \vphi{\varphi}

\newcommand \om{\omega}
\newcommand \Om{\Omega}

\newcommand{\ppb}{\partial\bar\partial}

\newcommand{\db}{\bar\partial}
\newcommand{\idd}{\frac{i}{2}dwd\bar{w}}
\newcommand{\ppbJ}{\partial_J\bar\partial_J}
\newcommand{\ppbJz}{\partial_{J_0}\bar\partial_{J_0}}
\newcommand{\ppbJP}{\partial_{J_P}\bar\partial_{J_P}}
\newcommand{\sumover}[2]{\sum_{i = {#1}}^{#2}}
\newcommand{\dV}{\mathrm{dvol}}
\newcommand{\dVd}{\mathrm{dvol}_D}

\newcommand{\Image}{\mathrm{Im}}

\renewcommand{\leq}{\leqslant}
\renewcommand{\geq}{\geqslant}

\begin{document}

\title{A Holomorphic Splitting Theorem}
\author{MIAO SONG}

\begin{abstract}

A long-term project is to construct a complete Calabi-Yau metric on the complement of the anticanonical divisor in a compact Kähler manifold $\oM$. We focus on the case where this smooth divisor has multiplicity 2 and is itself a compact Calabi-Yau manifold. Firstly we solved the Monge-Ampère equation when the Ricci potiential is of $O(r^{-1})$ decay on the generalized $ALG$ manifolds. Then we used the solution to this Kähler Ricci flat metric to prove a holomorphic splitting theorem: If $K_{\oM}=\calo(-2D)$, where $D$ can be realized as a smooth Calabi-Yau manifold, and if $\calo_{3D}(D)$ is trivial, then this Kähler manifold $\oM$ is biholomorphic to $\bbp^1\times D$.
\end{abstract}

\maketitle

\tableofcontents
\section{Introduction}


In 1990 Tian and Yau proposed a long-term program to construct complete Kähler Ricci-flat metrics from pairs $(\oM, D)$, where $\oM$ is a compact Kähler manifold and $D$ is a divisor such that $K_{\oM} = \calo(-\beta D)$. 
When we consider $M:=\oM\setminus D$, it is clear that there exists a global nowhere vanishing holomorphic volumn form. One can then construct a suitable background Kähler metric $\om$ on $M$, and the existence of a Kähler Ricci-flat metric is equivalent to solving a complex Monge-Ampère equation. 

In \cite{T-Y-1990} and \cite{T-Y-1991} they mainly dealt with the cases when 

\begin{enumerate}
    \item $\beta = 1$ and $D$ is neat and almost ample;
    \item $\beta > 1$ and $D$ admits a postive Kähler Einstein metric.
\end{enumerate}
In both settings, they established the existence of complete Ricci-flat Kähler metrics on $M$. They also addressed the case where $\oM$ is a fibration over a smooth algebraic curve $S$, with connected fibers, and the anticanonical divisor consisting of finitely many smooth reduced fibers. 

Subsequent work has explored a broader range of pairs $(\oM, D)$, as well as the asymptotic behavior of the resulting complete metrics. For instance, in \cite{C-H-2015}, Conlon and Hein constructed complete Kähler Ricci-flat metrics when $\beta > 1$ and $D$ admits a positive Kähler-Einstein metric or the $\mathbb{S}^1$-bundle of the normal bundle of $D$ admits an irregular Sasaki-Einstein structure. They also showed that the resulting metric is asymptotically conical. 

In \cite{C-L-22}, Collins and Li constructed new complete Calabi-Yau metrics on the complement of anticanonical divisors $D$, under the assumption that the normal bundle of $D$ is ample and that $D$ consists of two transversely intersecting smooth components. The asymptotic geometry in their setting is modeled on a generalization of the Calabi ansatz.

In \cite{HHN}, Haskins, Hein, and Nordström provided a new proof of the existence of asymptotically cylindrical Calabi-Yau manifolds. They considered pairs $(\oM, D)$ satisfying $K_{\oM} = \calo(-D)$, where the normal bundle of $D$ is holomorphically trivial and $D$ admits a Kähler Ricci-flat metric. Their paper also addresses the orbifold setting, where $D$ is an effective orbifold divisor and the orbifold normal bundle of $D$ is biholomorphic to $(\bbc\times D')/\langle\sigma\rangle$, with $D'$ smooth and equipped with a $\sigma$ invariant Kähler Ricci-flat metric. Here $\langle\sigma\rangle$ denotes the finite group of order $m$ acting by $\sigma(w,x)=(\mathrm{exp}(\frac{2\pi i}{m}w),\sigma(x))$ on the product. A key improvement in their work was the demonstration that the complete metric decays exponentially to a product Ricci-flat metric on the cylindrical end.

In this paper, we focus on a setting analogous to \cite{HHN}. But in stead of assuming $K_{\oM} = \calo(-D)$(i.e.,$\beta=1$), we consider the case $K_{\oM} = \calo(-2D)$(i.e.,$\beta=2$). Before stating our main result, we first clarify our definition of Calabi-Yau manifolds.

\begin{defn}
    A Kähler manifold $(M, g, J)$ is called a Calabi-Yau manifold if its canonical bundle $K_M$ is holomorphically trivial. This is equivalent to the existence of a nowhere vanishing holomorphic $(n,0)$-form $\Om$ on $M$. 
\end{defn}

This definition includes all complex tori, meaning we do not require $M$ to be simply connected or algebraic. Under this definition, we now state our main theorem:

\begin{thm}\label{splitting}
    Let $(\oM, \omega, J)$ be a compact Kähler manifold and $D$ be a smooth divisor in $\oM$ such that 
    \begin{enumerate}
        \item $K_{\oM} = \mathcal{O}(-2D)$;
        \item $D$ is a compact Calabi-Yau manifold;
        \item $\mathcal{O}_{3D}(D)$ is trivial as a holomorphic line bundle.
    \end{enumerate}
    Then $\oM$ is biholomorphic to $\mathbb{P}^1\times D$.
\end{thm}

Here $mD$ denotes the complex analytic space $(D, \calo/\cali^{m})$, where $\cali$ is the ideal sheaf of $D$ in $\oM$. Then $\mathcal{O}_{3D}(D)$ should be interpreted as the restriction of $\calo(D)$ to the second order infinitesimal neighborhood of $D$. Also one can view this as the quotient sheaf given by 
\begin{align*}
    0\to \cali^{m}\to \calo(D)\to \calo_{(m+1)D}(D) \to 0.
\end{align*}
We will mainly follow the discussion in \cite{HHN} to give more precise analytic interpretations of the third condition in Section 2. In the case when $h^1(\oM)=0$, the third condition is automatically satisfied.

As in prior work, our first step is to solve the complex Monge–Ampère equation on $\oM\setminus D$. The third condition plays a critical role in constructing suitable background metrics. A natural question is whether there exist concrete examples satisfying the hypotheses of theorem \ref{splitting}. Surprisingly, using the existence of the Ricci-flat Kähler metric, we prove a negative result: no nontrivial such examples exist.

We give a quick proof of this theorem in complex dimension two, where Enriques–Kodaira classification of compact Kähler surfaces applies.

\begin{proof}[proof when $\mathrm{dim}_{\bbc}\oM = 2$:]
    Since $K_{\oM} = \calo(-2D)$, the canonical divisor is divisible by 2. This implies that $\oM$ should be minimal and $k(\oM) = -\infty$. Therefore, $\oM$ must be either $\bbp^2$ or $\bbp(\mathcal{E})$, where $\mathcal{E}$ is some rank 2 holomorphic vector bundle over a Riemann surface $\Sigma_g$ (when $g=1$ they are Hirzebruch surfaces). It is clear that $\bbp^2$ is not possible. In the remaining cases, we have a holomorphic projection $\pi:\bbp(\mathcal{E})\to\Sigma_g$. The canonical bundle is given by
    \begin{align*}
        K_{\oM} = -2C + (k + e)f,
    \end{align*}
    where $kf$ is the pull-back of the canonical divisor on $\Sigma_g$ and $ef$ is the pull-back of the corresponding divisor of $\Lambda^2\mathcal{E}$ on $\Sigma_g$. $C$ can be viewed as the divisor added at infinity. 

    Since the only compact Calabi-Yau manifold of complex dimension 1 is complex torus, it follows that $\Lambda^2\mathcal{E}$ is trivial and $\Sigma_g = \bbc / \Lambda$. There are only two such types of $\mathcal{E}$ on complex torus, one is the direct sum of two trivial line bundles, whose projective bundle is biholomorphic to $\bbp^1\times \bbc/\Lambda$. The other is the non-trivial extension given by
    \begin{align*}
        0\to \calo\to\mathcal{E}\to\calo\to 0
    \end{align*}
    One can check that in this example $\calo_{2D}(D)$ is not trivial. More precisely, $\calo_{2D}(D)$ corresponds to the nontrivial extension element in $\mathrm{Ext}^1(\calo, \calo)$.
\end{proof}

From the proof, we see that there are examples where $\calo_{2D}(D)$ is nontrivial, and $\oM$ is not a product holomorphically. It remains unclear to the author whether condition (3) in theorem \ref{splitting} is indeed necessary.

The proof of theorem \ref{splitting} proceeds in the following steps. First, we construct a complete background Kähler metric on $M=\oM\setminus D$ that asymptotically approaches the product metric on $\bbc\times D$ at infinity, with decay rate $O(r^{-1})$. Next, we refine this metric so that the Ricci potential $f=O(r^{-\mu})$ for some $\mu > 2$, and $\int (e^{f}-1)\om^n=0$. This step involvs solving Poisson equations the complete Kähler manifolds. Finally, we establish the existence of a complete Calabi–Yau metric on $M$, together with a holomorphic function $w$ satisfying $\text{Hess}w=0$. This implies that $M$ is biholomorphic to $\bbc\times D$, with the biholomorphism given by $w$.

\textbf{Acknowledgements} The author would like to thank his advisor, Professor Xiuxiong Chen. The author was partially supported by Simons Foundation International, LTD.

\section{Weighted Spaces}
In order to finally solve the complex Monge-Ampère equation on the noncompact manifold $M=\oM\setminus D$, we first establish the weighted spaces and a weighted Sobolev inequality. However, our definitions of these weighted spaces are slightly different from those in \cite{CC-21}, which makes it harder to apply some known methods to solve the Poisson equation and the complex Monge-Ampère equation. This main difference comes from the analysis on the local complex structures of $D$ in a tubular neighborhood $U$.

\subsection{Motivation of new weighted spaces}
Let $\oM$ be a complex K\"ahler manifold with complex dimension $n\geqslant 2$. Suppose $D$ is a smooth divisor in $\oM$ such that it is also a smooth Calabi-Yau manifold and $2D\in|-K_{\oM}|$. From the adjunction formula we know that the normal bundle of $D$ in $\oM$ is holomorphically trivial. Our goal is to describe the local structure of the holomorphic volume form $\Om$ and the complex structure near the divisor $D$.

If the normal bundle $\caln$ of $D$ is just complex trivial, then we can find a tubular neighborhood $U$ of $D$ and a differeomorphism $\Phi : \Delta \times D \rightarrow U$ such that $\Phi (0,q)=q$ for $q\in D$ and $\Phi^* J-J_0=0$ along $\{0\}\times D$. Here $\Delta$ is the unit disk in $\bbc$ and $J_0$ is the product complex structure. However, even if $N$ is holomorphically trivial, the diffeomorphism $\Phi$ need not to be biholomorphic. In particular, the defining function $w$ for $D$ in $U$, induced by this $C^{\infty}$ trivialization, may not be holomorphic. 

In \cite{HHN} they used the implicit function theorem to modify $\Phi$ (via composition with a suitable diffeomorphism $F\in \mathrm{Diff(U)}$) so that the $J_0$-holomorphic disks $\Delta\times \{q\}$ becomes $J$-holomorphic with respect to $(F\circ\Phi)^{-1}$. Under this new product structure $\Delta\times D$, we can decompose the tensor $\Phi^*J-J_0$ into four components. This gauge-fixing procedure ensures that the components in $T^*\Delta\otimes(T\Delta\oplus TD)$ vanish. We denote the components of $\Phi^*J-J_0$ in $T^*D\otimes T\Delta$ and $T^*D\otimes TD$ by $K$ and $L$, respectively. Clearly, $L$ vanishes to first order along $D$. The next lemma, proved in \cite{HHN}, shows that the vanishing order of $K$ is determined by the restriction of $\calo(D)$ to an infinitesimal neighborhood of $D$ in $U$. 

\begin{lemma}\label{lemma 2.1}

Let $z$ be the holomorphic projection $\Delta\times D\to \Delta$. Then $N$ is trivial as a holomorphic line bundle if and only if the $T^*D\otimes T\Delta$ component of $\Phi_* J-J_0$ is $O_{g_0}(|z|^2)$. In particular, $\bar{\partial}z=\frac{i}{2}K\circ dz=O_{g_0}(|z|^2)$. If $\calo_{mD}(D)$ is holomorphically trivial, then we can choose the defining function $w$ such that $\bar{\partial}z=O_{g_0}(|z|^{m+1})$. 

\end{lemma} 

Here $g_0 = idzd\bar{z} + \om_D$ denotes the product metric on $\Delta\times D$, with $\om_D$ any fix metric on $D$. Since $\Phi^* J-J_0$ and $\bar{\partial}z$ are tensors on $U$, the notation $O_{g_0}(|z|^k)$ indicates that their $C^0$-norms, measures with respect to $g_0$, decay like $|z|^k$. 

Let $w=\frac{1}{z}$, then we obtain a holomorphic map $\Delta^*\times D\to B^c_1(0)\times D\subset\bbc\times D$. We denote $g_{\infty} = idwd\bar{w} + \om_D$ to be the new product metric on this chart. Let $r=|w|$ be the Euclidean distance from the origin in $\bbc$. Then $r$ is comparable to the distance measured using $g_{\infty}$ when $r$ is large enough. We let $O(\cdot) = O_{g_{\infty}}(\cdot)$ be the growth/decay estimates measured using the product metric $g_{\infty}$. It then follows that: 
\begin{align*}
    O_{g_0}(|z|^k) = O(r^{-k}),\,|dz| = \frac{1}{|w|^2}|dw| =O(r^{-2}),\, |dzd\bar{z}| =O(r^{-4})
\end{align*}

We refer to $\{\Delta\times D, (z,x)\}$ as the zero-chart  and to $\{B^c_R(0)\times D, (w,x)\}$ as the infinity-chart. Typically we use $g_0$ to measure the tensors on the zero-chart, and use $g_p:=g_{\infty}$ to measure tensors on the infinity-chart. 
If $\mathcal{O}_{3D}(D)$ is trivial, we obtain the following estimates:
\begin{align*}
    \metinf{K}=\metinf{\bar{\partial}z}=O(\frac{1}{r^2}),\, \metinf{L}=O(\frac{1}{r}),\,\metinf{J-J_0}=O(\frac{1}{r})
\end{align*}

\begin{rem}\label{rm4.1.2}
    From the standard deformation theory, if $b^1(D)=0$ and the normal bundle is holomorphically trivial, then $\mathcal{O}_{U}(D)$ is trivial for some small neighborhood $U$ of $D$. In particular, $\mathcal{O}_{mD}(D)$ is holomorphically trivial for all $m\in \bbz^+$, which implies that $z$ is a locally holomorphic defining function of $D$ in $U$. Moreover, if $b^1(\oM) = 0$, then one can construct a global holomorphic submersion $\oM\to \bbp^1$ such that $D$ is a generic fiber.
\end{rem}

In the following of this paper, we always assume that $\mathcal{O}_{3D}(D)$ is holomorphically trivial. 

From the diffeomorphism $\Phi$, we derive a $J-$holomorphic local chart $(U_{\alpha}, z', x'_1,...,x'_{n-1})$ covering $D$, where in each $U_{\alpha}$, 
\begin{align}
    z'=\sum_{k=1}^{\infty}f_k(x_j)z^k,\quad x'_i=x_i+A_i(z,x_j)z+B_i(z,x_j)\bar{z}
\end{align}
 for smooth functions $f_k,A_i$ and $B_i$.
 Here $(U_{\alpha}, z, x_1,...,x_{n-1})$ is a $J_P$-holomorphic chart, and $z$ is a global $J_P$-holomorphic function on a neighborhood of $D$. Thanks to Appendix A in \cite{HHN}, the triviality of $\mathcal{O}_{3D}(D)$ ensures that $f_1$ and $f_2$ are holomorphic functions in $D$. Let
 \begin{align}
     \Tilde{z}'=\frac{z'}{f_1(x_j')}\cdot(1-\frac{f_2(x_j')}{f_1(x_j')}z'),
 \end{align}
 then $\Tilde{z}'$ is holomorphic in $U_{\alp}$ and $(U_{\alpha}, \Tilde{z}', x'_1,...,x'_{n-1})$ is still a local chart. Now there is a smooth function $f(z,x_j)$ in $U_{\alp}$ such that $\Tilde{z}' = z + z^3f(z, x_j)$. We still use $(z',x'_1,...,x'_{n-1})$ to denote these refined local coordinates.

Let $\Om$ be the meromorphic volume form on $\oM$ with a double pole along $D$. In these local coordinates, we can write:
\begin{align}
    &\Om=\frac{h(z',x_j')}{z'^2}dz'\wedge dx'_1\wedge...\wedge dx'_{n-1}  \label{e2.1}
\end{align}
for some holomorphic function $h(z',x_j')$ in $U_{\alp}$. In particular, $h(0,x_j')$ never vanishes. Now define:
\begin{align}
    &\Om_0=\frac{h(x_j)}{z^2}dz\wedge dx_1\wedge...\wedge dx_{n-1}. \label{e2.2}
\end{align}
    Comparing $\Om$ and $\Om_0$, we have: 
\begin{align}
    \Om-\Om_0=\phi_1 + \frac{dz}{z^2}\wedge \phi_2 + d\bar{z}\wedge \phi_3 + \frac{dz\wedge d\bar{z}}{z^2}\phi_4
\end{align}
where 
\begin{align}
    \phi_1\in \Gamma(\Lambda^{n}_{\bbc}(D)), \phi_2\in \Gamma(\Lambda^{n-1}_{\bbc}(D)), \phi_3\in \Gamma(\Lambda^{n-1}_{\bbc}(D)), \phi_4\in \Gamma(\Lambda^{n-2}_{\bbc}(D))
\end{align} for given $z$. They are smooth in $U_{\alp}$. When measuring these forms using metric $g_0$, we have
\begin{align}
    \metzero{\phi_1} = O_{g_0}(z), \metzero{\phi_2} = O_{g_0}(z, \bar{z}), \metzero{\phi_3} = O_{g_0}(z), \metzero{\phi_4} = O_{g_0}(1)
\end{align}
Converting these estimates using $g_{\infty}$, and recalling that $|dz|=O(r^{-2})$, we find:
\begin{align}
    &\metinf{\phi_1} = O(\frac{1}{r}), \metinf{\frac{dz}{z^2}\wedge \phi_2} = O(\frac{1}{r}), \metinf{d\bar z\wedge \phi_3} = O(\frac{1}{r^3}), \\ &\metinf{\frac{dz\wedge d\bar{z}}{z^2}\phi_4} = O(\frac{1}{r^2})
\end{align}
which implies that 
\begin{align}
    \metinf{\Om - \Om_0} = O(\frac{1}{r}).
\end{align}
Using the holomorphic transition functions of the line bundle $K_{\oM}$, one can see that $ (U_{\alp}\cap D, h(x_j)dx_1\wedge...\wedge dx_{n-1})$ defines a global nowhere vanishing section $R\in \Gamma(K_D)$. This means $\Om_0=\frac{1}{z^2}dz\wedge R$. It is now clear that the complex structure corresponds to $\Om_0$ is the product complex structure.

Then we have the following lemma:
\begin{lemma}
    There exists $f_1, f_2\in C^{\infty}(D;\bbc), f_3\in C^{\infty}(D;\bbr), f_4\in C^{\infty}(M)$ such that 
    \begin{align}
        \frac{\Om\wedge\overline{\Om}}{\Om_0\wedge\overline{\Om_0}} = (1 + 2Re(zf_1) + 2Re(z^2f_2) + |z|^2f_3 + f_4)
    \end{align}
    where $f_4=O_{g_0}(|z|^3)$
\end{lemma}

Using these local coordinates, since $\bar{\partial}z'=0$ and $\db x_i'=0$, we can derive 
\begin{align}
    \metinf{\nabla^k_{\bbr^2}\nabla^l_D\db z}\leq C_kr^{-k-2},\;\metinf{\nabla^k_{\bbr^2}\nabla^l_D\db x_i}\leq C_kr^{-k-1}.
\end{align}
Using the fact that $\db z = \frac{i}{2}K\circ dz$ and $\db x_i=\frac{i}{2}L\circ dx_i$, we can also get 
\begin{align}
    |\nabla^k_{\bbr^2}\nabla^l_D(J-J_0)|_{g_{\infty}} \leq C_k r^{-1-k},\;{|\nabla^k_{\bbr^2}\nabla^l_D(\Om - \Om_0)|_{g_{\infty}} \leq C_k r^{-1-k}}.
\end{align}

\begin{rem}
    We can construct a Kähler structure on $\Delta\times \bbt^2$ as follows. Let $dx$ be the holomorphic volumn form on $\bbt^2$ and $z$ be the holomorphic coordinate function on $\Delta$. Let $\Om = \frac{1}{z^2}dz\wedge(dx + zd\bar{x})$ be the meromorphic volumn form. It is easy to see that $d\Om = 0$. Then this defines a complex structure on $\Delta_{|z|\ll 1}\times \bbt^2$, which is the same as the complex structure defined by $dz\wedge(dx + zd\bar{x})$. Since $dz\wedge\Om =0$, $z$ is then a global holomorphic function, which means that $z:\Delta\times D\to \Delta$ is a holomorphic submersion. Let $D_z = \{z\}\times \bbt^2$, then $\Om$ has a double pole along $D_0$. The complex structure on $D_z$ is defined by $dx+zd\bar{x}$, which are different from each other. Let $\om = idz\wedge d\bar{z}+idx\wedge d\bar{x}$. Then $\om\wedge \Om = 0$, which means that $(M = \Delta\times\bbt^2, \om, \Om)$ is a Kähler structure. According to remark \ref{rm4.1.2}, we have $\calo_{mD_0}(D_0)$ that are trivial for all $m\in \bbz_{\geq 1}$. However, the complex structure is not a product complex structure. 
\end{rem}
\subsection{Generalized ALG manifolds}
We begin by introducing a notion of generalized ALG manifolds in higher dimensions. These should be regarded as higher-dimensional analogues of four-dimensional ALG Kähler manifolds.

\begin{defn}
    A Kähler manifold $(M, g, J)$ is called weakly $\text{ALG}_{\beta}(\frac{2\pi}{n})$ if there is a compact subset $K\subset M$, a compact Calabi-Yau manifold $(D, g_D, J_D)$, a ball $B_R(0)\subset \bbc$ with radius $R$, and a cyclic group $\langle \sigma\rangle$ with order $n$ satisfying:
    \begin{enumerate}
        \item $\sigma$ acts on $\bbr^2=\bbc$ by multiplying $e^{i\frac{2\pi}{n}}$. And the action of $\sigma$ on $(D, g_D, J_D)$ is nontrivial and preserve the Kähler structure,
        \item there is a diffeomorphism $\Psi: M\setminus K\to (B_R^c(0)\times D)/\langle \sigma\rangle$, where $B^c_R(0)\subset \bbc$,
        \item assuming $g_p=g_{\bbc} + g_D$ and $J_P = J_{\bbc} + J_D$, then 
        \begin{align*}
            |\Psi_*g - g_P|_{g_P} = O(r^{-\beta}), \quad |\Psi_*J - J_P|_{g_P} = O(r^{-\beta})
        \end{align*}
        on $\Psi(M\setminus K)$
    \end{enumerate}
\end{defn}

\begin{defn}
    for any $\delta>0$, a tensor $T$ on the end $E=\eb{R}{D}$ is in the class $\calc^{k,\alpha}_{\delta}$ if:
    \begin{enumerate}
        \item $T\in C^{k,\alpha}(E)$;
        \item for each $l,m\in\bbn_{\geq 0}$, there is a constant $C_{l,m}$ such that $|\nabla_{\bbr^2}^l\nabla_D^m|\leq C_{l,m}\frac{1}{r^{\delta+l}}$.

    \end{enumerate}
    Here, the derivative is taken using the product metric $g_P$, so does the $C^{k,\alp}$ norms. $r$ is the distant function on $\bbr^2$, and when $R\gg 1$, it is comparable to the distant function $d(p, p_0)$ on $M$ using $g_P$, where $p_0$ is some point in $K$.
\end{defn} 

We denote the space $\calc^{\infty}_{\delta}(E) = \bigcap_{k\leq 0}\calc^{k,\alp}_{\delta}(E)$. We use $\calc^{k,\alp}_{\delta}(\calc^{\infty}_{\delta})$ to denote $C^{k,\alp}(C^{\infty})$ functions on $M$ such that they are in $C^{k,\alp}_{\text{loc}}(M)(C^{\infty}(M))$and when restricted on $E$, they are in the classes $\calc^{k,\alp}_{\delta}(\calc^{\infty}_{\delta})$. For convenience, we use $\estp{r^{-\delta}}$ to denote a function in $\calc^{\infty}_{\delta}$. 
An example of a function $f\in \calc_{\delta}^{\infty}$ can be constructed as follows: Let $(M, g) = (\bbr^2\times D, g_{\bbr^2}+g_D)$, $\chi(r)$ be a smooth function on $\bbr^2$ such that $\chi(r)=0$ in $B_1(0)$ and $\chi(r)=1$ in $B_2^c(0)$. Let $v$ be a smooth function on $D$, then $f=\chi(r)\cdot\frac{1}{r^\delta}\cdot v$ is in the class $\calc_{\delta}^{\infty}$ for $\delta>0$.

\begin{defn}
    A weakly $\ALGb(\frac{2\pi}{n})$ Kähler manifold $(M, g, J)$ is said to be an $\ALGb(\frac{2\pi}{n})$ manifold if, at the end, we have: 
    \begin{align}
        \Psi_*g-g_P\in \calc^{\infty}_{\beta},\quad \Psi_*J-J_P\in \calc^{\infty}_{\beta}.
    \end{align}
    If $n=1$, we simply call it an $\ALGb$ manifold. 
\end{defn}

Now introduce polar coordinate $w=e^{t+i\theta}$, and define a new product metric $g'_{\infty} = dt^2+d\theta^2+g_D$. Then for any $\ALGb(\frac{2\pi}{n})$ manifold $(M, J)$ we have $|\nabla_{g'_{\infty}}^k(\Psi_*J-J_P)| = O(e^{-\beta t})$. When $n=1$, using the gauge-fixing method in \cite{HHN} we can derive the following lemma:

\begin{lemma}
    One can choose the diffeomorphism $\Psi$ such that $\Psi_*J - J_P = K+L$, where $K\in\Gamma(T(\Delta)\otimes T^*(D))$ and $L\in\Gamma(T(D)\otimes T^*(D))$. Meanwhile, $K,L \in \calc_{\beta}^{\infty}$. 
\end{lemma}
When $n>1$, one can lift to the covering space $\Psi^*(B_R^c(0)\times D)$ of the end. This lemma shows that for $\ALGb(\frac{2\pi}{n})$ manifolds we can always assume that $\Psi_*J-J_P=K+L$. Moreover, these manifolds can be compactified to a Kähler orbifold. Then the decay rate of the tensor $K$ closely relates to the sheaf $\calo_{mD}(D)$. We will later prove that in our cases where $M=\oM\setminus D$, we can design a Kähler metric on $M$ such that $(M, g, J)$ is a $\ALGone$ manifold and $K\in \calc_2^{\infty}$. 
For an $\text{ALG}_{\beta}(\frac{2\pi}{n})$ manifold we fix an end $E=M\setminus K$ and the corresponding diffeomorphism $\Psi$. For simplicity, we will often omit the diffeomorphism and perform calculations directly on $(B_R^c(0)\times D)/\langle \sigma\rangle$, assuming $R$ is sufficiently large.

For convenience, we give the following definition:

\begin{defn}
    A Kähler manifold $(M,g,J)$ is called $\ALGog$ if it is $\ALGone$ and the component $K\in \calc_{\gamma}^{\infty}$ for some $\gamma > 1$.
\end{defn}

\subsection{A Weighted Sobolev Inequality}
In this subsection, we prove a weighted Sobolev inequality on the product of an annulus and a compact manifold. One can verify that this inequality also holds if the end of the manifold is $\bbc\times D/\langle\sigma\rangle$.
\begin{prop}
    Let $g_P$ be the product metric on $M$. $g$ is another complete metric on $M$ such that $g-g_P = O(\frac{1}{r^{\beta}})$, where $r$ is the distance function with respect to $g_p$. Fix a number $\lambda > 1$, and define $A_i = A(\lambda^i, \lambda^{i+1})\times Y$, where $A(r,R)$ denotes the annulus in $\bbr^2$ with inner radius and outer radius $r, R$, respectively. Then there exists a uniform constant $C$ such that for all $u\in C^1(M)$ and $\sigma\in [1, \frac{n}{n-2}]$, we have
    \begin{align*}
        ||u-u_i||_{2\sigma,g,A_i}\leq C(\lambda^i)^{\frac{1}{\sigma}}||\nabla u||_{2,g,A_i}.
    \end{align*} 

    Moreover, if $v\in C^{\infty}_0(A_i)$, then
    \begin{align*}
        ||v||_{2\sigma,g,A_i}\leq C(\lambda^i)^{\frac{1}{\sigma}}||\nabla v||_{2,g,A_i}
    \end{align*}
\end{prop}

\begin{proof}
    Let $\vphi_i: A_0\to A_i$ be the diffeomorphism satisfying $\vphi(z',q) = (\lambda^iz',q)$. Denote $g_0$ as the product metric on $A_0$. Then we have $\vphi_i^*g = \lambda^{2i}(dx'^2 + dy'^2) + g_Y + h_1 + h_2 + h_3$, where $h_1, h_3$ are symmetric tensors on $A(1,\lambda)$ and $Y$, respectively. $h_2$ is the mixed term. We also have the estimate on these error terms:
    \begin{align*}
        ||h_1||_{g_0} = O(\lambda^{2i-\beta i}), \quad ||h_2||_{g_0} = O(\lambda^{i-\beta i}), \quad ||h_3||_{g_0} = O(\lambda^{-\beta i})
    \end{align*}
    We construct a new metric $g_{0,i} = dx'^2 + dy'^2 + g_Y + \lambda^{-2i}h_1 + \lambda^{-i}h_2 + h_3$. Then we have $\vphi_i^* \dV_g = \lambda^{2i}\dV_{g_{0,i}}$ on $A_0$. Since the error terms are uniformly controlled, these $g_{0,i}$ are uniformly equivalent to $g_0$ on $A_0$. This implies that the Sobolev constants are uniform with respect to these metrics.

    Let $\bar{u_i}$ be the average of $u$ on $A_i$ with respect to $g$. Denote $u'(z', q) = u(\lambda^i z',q)$. Then
    \begin{align*}
        \int_{A_i}u(z,q)\dV_g &= \int_{A_0}u(\lambda^i z',q)\vphi^*\dV_g   \\
        &=\lambda^{2i}\int_{A_0}u'(z',q)\dV_{g_0,i}.
    \end{align*}
    Take $u\equiv 1$ we have $|A_i|_g=|A_0|_{g_{0,i}}$. So $\bar{u_i} = \bar{u'}_{g_{0,i}}$.

    Since $du' = \partial_z u\cdot \lambda^idz' + \partial_qu\cdot dq$, there is a uniform constant $C$ such that ${|du'|^2_{g_{0,i}} \leq C \lambda^{2i}|du|^2_g}$, so
    \begin{align*}
        \int_{A_0}|du'|^2_{g_{0,i}}\dV_{g_0,i} \leq C\int_{A_0}|du|^2_g\lambda^{2i}\dV_{g_0,i}=C\int_{A_i}|du|^2_g\dV_g
    \end{align*}
    
    Using the Sobolev inequality on $A_0$ we can derive that 
    \begin{align*}
        (\int_{A_i}|u - \bar{u_i}|^{2\sigma}\dV_g)^{\frac{1}{2\sigma}} &= (\int_{A_0}|u' - \bar{u'_{g_0,i}}|^{2\sigma}\lambda^{2i}\dV_{g_{0,i}})^{\frac{1}{2\sigma}}  \\
        &\leq C\lambda^{\frac{i}{\sigma}}(\int_{A_0}|du'|^2_{g_{0,i}}\dV_{g_0,i})^{\frac{1}{2}} \\
        &\leq C \lambda^{\frac{i}{\sigma}}(\int_{A_i}|du|^2_g \dV_g)^{\frac{1}{2}}.
    \end{align*}
    This gives the proof of the proposition.

    The proof on the Sobolev inequality for compact supported functions follows the same argument.

\end{proof}

\begin{prop}\label{sobolev}
    For any $\ALGb$ manifold $(M,g,J)$ and any $\eps > 0$, there exists a piecewise constant positive function $\psi_{\eps}=O(r^{-2-2\eps})$, such that $\int \psi_{\eps}=1$ and for all $\alpha\in [1, \frac{n}{n-2}]$, $u\in C^{\infty}_0(M)$
    \begin{align}
        ||r^{-1-\eps}(u-\bar{u}_{\eps})||_{2\alpha}^2\leq C(\eps)||\nabla u||_2^2
    \end{align}
    where $u_{\eps}=\int u\psi_{\eps}$.
\end{prop}

\begin{proof}
    Take $A_0=K$ being a compact subset of $M$ such that $M\setminus K$ is asymptotical to $B_{\lambda}^c(0)\times Y$ for some $\lambda> 1$. Let $r_i=\lambda^i$ and $A_i=B_{i+1}(0)\setminus B_{i} \times Y$. Then $\bigcup_{i\geq 0}A_i=M$. For each $i>0$, by the previous proposition, we have 
    \begin{align}
        ||r^{-1-\eps}(u-\bar{u}_i)||_{2\alpha,A_i}^2\leq r_i^{-2-2\eps}||(u-\bar{u}_i)||_{2\alpha,A_i}^2\leq C||\nabla u||_{2,A_i}^2
    \end{align}
    where $\bar{u}_i$ is the average of $u$ over $A_i$ and $C$ is uniform.

    Now take $\chi_i$ be the characteristic function of $A_i$,
    \begin{align}
        ||r^{-1-\eps}(u-\bar{u}_{\eps})||_{2\alpha}^2
        &\leq \sum (r_i^{-2-2\eps}||\chi_i(u-\bar{u}_i)||_{2\eps}^2+r_i^{-2-2\eps}|\bar{u}_i-\bar{u}_{\eps}|^2)      \nonumber\\
        &\leq C||\chi_i\nabla u||_{2}^2+\sum r_i^{-2-2\eps}|\bar{u}_i-\bar{u}_{\eps}|^2
    \end{align}
    Let $\psi_{\eps}=\sum \frac{1}{|A_i|}r_i^{-2-2\eps}\chi_i/\sum r_i^{-2-2\eps}$. It is easy to check $\int \psi_{\eps}=1$ and $\sum r_i^{-2-2\eps}(\bar{u}_i-\bar{u}_{\eps})=0$.
    \begin{align}
        \sum r_i^{-2-2\eps}|\bar{u}_i-\bar{u}_{\eps}|^2
        &\leq C\sum_{i<j} (r_ir_j)^{-2-2\eps}|\bar{u}_i-\bar{u}_j|^2    \nonumber\\
        &\leq C\sum_{i<j} (r_ir_j)^{-2-2\eps}|i-j|\sum_{k=i}^{j-1}|\bar{u}_k-\bar{u}_{k+1}|^2
    \end{align}
    Denote $B_k=A_k\cup A_{k+1}$, then
    \begin{align}
        |\bar{u}_k-\bar{u}_{k+1}|^2\leq \frac{1}{|A_k||A_{k+1}|}\int_{A_k\times A_{k+1}}|u(x,p)-u(y,q)|^2
        &\leq \frac{2|B_k|}{|A_k||A_{k+1}|}\int_{B_k}|u-\bar{u}_{B_k}|^2    \nonumber\\
        &\leq C||\nabla u||_2^2
    \end{align}
    Using the fact that $\sum_{i<j} (\lambda^{i+j})^{-2-2\eps}|i-j|\leq \sum_jj^2(\lambda^j)^{-2-2\eps}=C(\eps)<\infty$, we get the weighted Sobolev inequality.
\end{proof}

The weighted Sobolev inequality also holds for real Riemannian manifolds which satisfies the $\calc_{\beta}^{\infty}$ decaying condition on the metric. Using this proposition, we can solve the Poisson equation on $(M,g,J)$ in which $f$ has a relatively faster decay and an integrability condition. See proposition \ref{ALGpoisson}.

\section{Linear Theory}
In this section we consider the Poisson equations $\Delta u = f$ on noncompact 
$\ALGot$ manifold $(M, g)$. 
To begin with, let us assume that $(M, g)=(\bbr^2, g_{\text{flat}})$, and $r$ is the Euclidean distance from the point to the origin. Meanwhile, we assume that for every $k\in\bbn_{\geq 0}$,

\begin{align}\label{cond3.1}
    |\nabla^kf|\leq C_kr^{-\mu-k}
\end{align}
for some $\mu \geq 1$. Firstly, we define some spaces in which we can have good controls on the behavior of functions when $r\to +\infty$.

\begin{defn}\label{d3.1}
    Let $u\in C^{\infty}(\bbr^2)$ and $r\gg 1$, we say 
    \begin{enumerate}
        \item $u\in \cald^{\infty}_{r\log r}(\bbr^2)$ if $u=\estim{r\log r}$, $u=\estim{\log r}$ and $|\nabla^k u|\leq C_kr^{1-k}$;
        \item $u\in \cald^{\infty}_{r^{\delta}}(\bbr^2)$ if $|\nabla^k u|\leq C_kr^{\delta-k}$;
        \item $u\in \cald^{\infty}_{(\log r)^2}(\bbr^2)$ if $u=\estim{(\log r)^2}$, $|\nabla^k u| \leq C_kr^{-k}\log r$ for $k\geq 1$;
        \item $u\in \cald^{\infty}_{\log r}(\bbr^2)$ if $u= O(\log r)$ and $|\nabla^k u| \leq C_kr^{-k}$ for $k\geq 1$.
    \end{enumerate}
\end{defn}

These classes will be used to describe the decay or growth rate of the solution $u$ to the Poisson equation. We will give the proof to proposition \ref{prop3.2} in the Appendix A.
\begin{prop}\label{prop3.2}
    Let $f\in C^{\infty}(\bbr^2)$ and $f$ satisfies condition \ref{cond3.1} for some $\mu \geq 1$. Then there is a solution $u$ to equation $\Delta u=f$. The solution is unique up to linear functions. Moreover, 
    \begin{enumerate}
        \item if $\mu = 1$, then $u\in\cald^{\infty}_{r\log r}(\bbr^2)$;
        \item if $\mu\in (1, 2)$, then $u\in\cald^{\infty}_{r^{2-\mu}}(\bbr^2)$
        \item if $\mu=2$, then $u\in\cald^{\infty}_{(\log r)^2}(\bbr^2)$
        \item if $\mu > 2$, then $u\in\cald^{\infty}_{\log r}(\bbr^2)$
    \end{enumerate}
\end{prop}

\begin{rem}
    We give some examples to illustrate this phenomenon:
    \begin{enumerate}
        \item Let $\Tilde{u}=re^{i\theta}\log r$ and $\chi(r)$ be a smooth function such that $\chi(r)=0$ when $r\leq 1$, $\chi(r)=1$ when $r\geq 2$. Then let $u=\chi(r)\cdot\Tilde{u}$. Now $u$ is a smooth function on $\bbr^2$ and $\Delta u= \frac{e^{i\theta}}{r}$ when $r>1$. The function $\text{Re}(\frac{e^{i\theta}}{r})=\frac{x}{r^2}$ is the primary obstacle in finding the complete Calabi-Yau metric in this paper.
        \item Since $\Delta (\log r)^2 = \frac{2}{r^2}$, this gives the evidence of case(3). And $\Delta \log r = 0$ shows that even if $f$ is compactly supported, the solution $u$ might still grow like $\log r$.
    \end{enumerate}
\end{rem}

\subsection{Solving Poisson Equation on Product Manifolds}
In this subsection, we consider the Poisson equation $\Delta u=f$ on a product manifold $M=\bbr^2\times Y$, where $Y$ is a compact manifold with a Riemannian metric $g_Y$. The Laplacian is taken with respect to the product metric.

We begin by defining the fiberwise average of $f$ over $Y$:
\begin{align}
    f_1(z) = \int_{\{z\}\times Y}f(z,q)\dV_Y,
\end{align}
 where $(z,q)\in \bbr^2\times Y$. A solution to $\Delta u_1 = f_1$ is given in proposition \ref{prop3.2}. Now we assume that $f\in C_0^{\infty}(M)$ is a smooth function on the product manifold satisfying the fiberwise mean-zero condition:
\begin{align}\label{c2.4}
    \int_{\{z\}\times Y}f(z,q)\dV_Y=0 \quad \text{for all }z\in \bbr^2.
\end{align}

Let $\{v_i\}_{i=0}^{\infty}$ be a complete orthonormal basis of $L^2(Y)$ consisting of eigenfunctions of the Laplacian $\Delta_Y=g^{jk}_Y\partial_j\partial_k$, satisfying:
\begin{align*}
    \Delta_Y v_i=-\lambda_i v_i,\quad \lambda_0=0<\lambda_1\leq \lambda_2\leq\ldots.
\end{align*}
Then we decompose $f$ and $u$ as Fourier series along the $Y$ direction:
\begin{align*}
    f(z,q) = \sum_{i=1}^{\infty}f_i(z)v_i(q),\quad u(z,q)=\sum_{i=1}^{\infty}u_i(z)v_i(q),
\end{align*}
where $f_i(z)=\int_{\{z\}\times Y}fv_i\dV_Y$. The fiberwise mean-zero condition ensures that $f_0(z)=0$, so the expension starts from $i=1$. Plugging into the Poisson equation yields:
\begin{align}
    \Delta_{\bbr^2} u_i-\lambda_i u_i=f_i
\end{align}\label{c3.5}
The solutions are $u_i=-\frac{f*B_i}{2\pi}$ where $\widehat{B_i}=\frac{1}{\lambda_i+|\xi|^2}$. By using complex analysis, we can derive that $B_i=\frac{1}{2}\int_0^{\infty}\frac{e^{-\lambda_it-\frac{|z|^2}{4t}}}{t}dt$. Then the solution $u$ can be written into 
\begin{align}
u&=\sum_{i=1}^{+\infty}-\frac{1}{4\pi}(\int_{\bbc}\int_0^{\infty}\frac{e^{-\lambda_it-\frac{|z-w|^2}{4t}}}{t}\cdot f_i(y)\,dt\,\dV_{\bbc})v_i        \nonumber\\
 &=-\int_{\bbc}\int_Yf(y,q)\int_0^{\infty}\sum_{i=1}^{+\infty}e^{-\lambda_i t}v_i(p)v_i(q)\frac{e^{-\frac{|z-w|^2}{4t}}}{4\pi t}\,dt\,\dV_Y\dV_{\bbc}      \\
 &=\int_{\bbc\times Y}f(y,q)\int_0^{\infty}H_{\mathbb{R}^2}(z,w,t)H^0_Y(p,q,t)\,dt\,\dV_Y\dV_{\bbc} \nonumber
\end{align}
where $H_{\mathbb{R}^2}(z,w,t)$ is the heat kernel on $\mathbb{R}^2$ and $H^0_Y(p,q,t)$ is the heat kernel on $Y$ minus the volume of $Y$ with respect to the Kähler metric $\omega_Y$. We denote the Green's function
\begin{align}
    G((z,p),(w,q))=\int_0^{\infty}H_{\mathbb{R}^2}(z,w,t)H^0_Y(p,q,t)dt.
\end{align}
Then the solution to $\Delta u=f$ can be represented by $u=f*G$.

Let us first discuss the estimate of eigenfunctions $v_i$. Using intergration by part one can derive that $||v_i||_{2\sigma^{k+1}}\leq (|Y| + \frac{\sigma^k\sqrt{\lambda_i}}{\sqrt{2\sigma^k-1}})^{\frac{1}{\sigma^k}}||v_i||_{2\sigma^k}$. Here $|Y|$ is the volume of $Y$ given by the metric $g_Y$, which is a constant. Then the Moser iteration implies that $|v_i|\leq C\lambda_i^{\frac{n-2}{4}}$. Schauder estimates show that $|v_i|+\frac{1}{\sqrt{\lambda_i}}|\nabla v_i|+\frac{1}{\lambda_i}|\nabla^2v_i|\leq C'\lambda_i^{\frac{n-2}{4}}$. By Wely's law, $\lambda_i \sim i^{\frac{2}{n-2}}$ for $i\gg 1$. Since $\Delta_Y v_i = -\lambda_i v_i$, we also have 
\begin{align*}
    |f_i| = |\int_Yfv_i| = |-\frac{1}{\lambda_i}\int_Yf\Delta_Y v_i| = |\frac{1}{\lambda_i}\int_Y\Delta_Yfv_i| \leq C \frac{1}{\lambda_i^l} |\nabla^{2l}_Yf|\cdot|v_i|
\end{align*}
for any $l\in\bbn_{\geq 1}$. A standard argument shows that if $f\in C^{k,\alpha}_{\text{loc}}(U\times Y)$ with $k > \frac{3n-6}{2}$, then the Fourier series converges in $C^{2,\alpha}_{\text{loc}}(U\times Y)$. The same holds true for $u_i$. This shows that if $f\in C^{k,\alpha}$, then $u = f * G$ solves the Poisson equation $\Delta u = f$. 

\begin{lemma}
    Let $B_i(z)=\frac{1}{2}\int_0^{\infty}\frac{e^{-\lambda_it-\frac{|z|^2}{4t}}}{t}dt$. Then $0\leq B_i(z)\leq C\frac{e^{-\sqrt{\lambda_i}|z|}}{\sqrt{|z|}}$.
\end{lemma}
\begin{proof}
    We only need to prove the second inequality with $i=1$. The direct calculation shows that
    \begin{align*}
        \int_0^{\infty} \frac{e^{-\lambda_1t-\frac{|z|^2}{4t}}}{t}dt   
        &=2ce^{-\sqrt{\lambda_1}|z|}\int_0^{\infty}\frac{e^{-(\sqrt{\lambda_1t}-|z|/(2\sqrt{t}))^2}}{\sqrt{t}}d\sqrt{t}  \\
        &\leq 2ce^{-\sqrt{\lambda_1}|z|}\int_{-\infty}^{+\infty}\frac{e^{-\tau^2}}{\sqrt{\tau^2 + 2\sqrt{\lambda_1}|z|}}d\tau
    \end{align*}
    where we let $\tau=\sqrt{\lambda_1 t}-\frac{|z|}{2\sqrt{t}}$ 

    Since $\int_{-\infty}^{+\infty}\frac{e^{-\tau^2}}{\sqrt{\tau^2 + 2\sqrt{\lambda_1}|z|}}d\tau \leq \frac{1}{\sqrt{2\sqrt{\lambda_1}|z|}}\int_{-\infty}^{+\infty} e^{-\tau^2}d\tau$, we have
    \begin{align*}
        |\int_1^{\infty}H_{\mathbb{R}^2}(z,0,t)H^0_Y(p,p_0,t)dt|\leq c\frac{e^{-\sqrt{\lambda_1}|z|}}{\sqrt{|z|}}.
    \end{align*}
\end{proof}

\begin{lemma}
    $|\nabla ^k f|=O(r^{-\mu})$ implies $|\nabla^k u|=O(r^{-\mu})$.
\end{lemma}
\begin{proof}
    We only need to prove that $u=O(r^{-\mu})$ for the case when $k=1$. From the previous discussion, we know that $|f_i|_{C^2}\leq C\lambda_i^{-\frac{k}{2}}r^{-\mu}$. And $u_i(z) = f_i*B_i=\int_{\bbr^2}f_i(w)B_i(z-w)dw$. Now we assume $|z| \geq R$ for some large number $R$. Firstly we have 
    \begin{align*}
        |\int_{B_1(z)}f_i(w)B_i(z-w)dw|&\leq |\int _{B_1(z)}f_i(w)B_i(z-w)dw| \\
        &\leq C\lambda_i^{-\frac{k}{2}}R^{-\mu}|\int _{B_1(z)}\frac{1}{\sqrt{|z-w|}}w| \\
        &\leq C'\lambda_i^{-\frac{k}{2}}R^{-\mu}
    \end{align*}
    For the second part, we have 
    \begin{align*}
        |\int_{B_1^c(z)}f_i(w)B_i(z-w)dw|&\leq |\int_{B_1^c(z)\cap B_{R/2}(z)}f_i(w)B_i(z-w)dw| \\
        &+|\int_{B_1^c(z)\cap B^c_{R/2}(z)}f_i(w)B_i(z-w)dw|\\
        &\leq C_1\lambda_i^{-\frac{k}{2}}|\int_{\frac{R}{2}}^{+\infty}\frac{e^{-\sqrt{\lambda_i}r}}{\sqrt{r}}rdr| + C_2\lambda_i^{-\frac{k}{2}}R^{-\mu}|\int_1^{\frac{R}{2}}\frac{e^{-\sqrt{\lambda_i}r}}{\sqrt{r}}rdr|  \\
        &\leq C_3 \lambda_i^{-\frac{k}{2}}R^{-\mu}
    \end{align*}
    
    The higher order decaying rate can be derived from local Schauder estimates. 
\end{proof}

Now we can organize these calculations and arrive at the following property.
\begin{prop}
    Let $f\in C^{\infty}(M)$ such that $\int_{\{w\}\times Y}f\dV_Y = 0$ for all $w\in \bbr^2$. Then there is a solution to $\Delta u=f$. Moreover, if $|\nabla^k f|\in O(r^{-\mu})$, then $|\nabla^k u|\in O(r^{-\mu})$. The solution is unique if we require that $u$ is bounded.
\end{prop}

Actually one can also prove these estimates using the estimates on the Green's function. We will basically list the properties of $G$ here. 
\begin{lemma}
    When $d:=d((z,p))\to 0$, $|G((z,p),(0,p_0))|\leq c\frac{1}{d^{n-2}}$. When $d\to \infty$, $|G((z,p),(0,p_0))|\leq c\frac{e^{-\sqrt{\lambda_1}|z|}}{\sqrt{|z|}}$ and $d\approx |z|$.
\end{lemma}

\begin{proof}
The first asymptotical behavior comes from the standard argument on Green's function and the fact that $d^{n-2}\log d \to 0$ as $d\to 0$. 

For the second part, notice that on compact manifolds we have $\sum_ie^{-\lambda_it}v_i(p)v_i(q)\approx ce^{-\lambda_1t}$ when $t>1$, so we have
\begin{align*}
   |\int_1^{\infty}H_{\mathbb{R}^2}(z,0,t)H^0_Y(p,p_0,t)dt|&\leq c\int_1^{\infty} \frac{e^{-\lambda_1t-\frac{|z|^2}{4t}}}{t}dt   \\
   &=2ce^{-\sqrt{\lambda_1}|z|}\int_1^{\infty}\frac{e^{-(\sqrt{\lambda_1t}-|z|/(2\sqrt{t}))^2}}{\sqrt{t}}d\sqrt{t}  \\
   &\leq 2ce^{-\sqrt{\lambda_1}|z|}\int_{-\infty}^{+\infty}\frac{e^{-\tau^2}}{\sqrt{\tau^2 + 2\sqrt{\lambda_1}|z|}}d\tau
\end{align*}
where we let $\tau=\sqrt{\lambda_1 t}-\frac{|z|}{2\sqrt{t}}$ 

Since $\int_{-\infty}^{+\infty}\frac{e^{-\tau^2}}{\sqrt{\tau^2 + 2\sqrt{\lambda_1}|z|}}d\tau \leq \frac{1}{\sqrt{2\sqrt{\lambda_1}|z|}}\int_{-\infty}^{+\infty} e^{-\tau^2}d\tau$, we have
\begin{align*}
    |\int_1^{\infty}H_{\mathbb{R}^2}(z,0,t)H^0_Y(p,p_0,t)dt|\leq c\frac{e^{-\sqrt{\lambda_1}|z|}}{\sqrt{|z|}}.
\end{align*}

On the other hand, $|\sum_ie^{-\lambda_it}v_i(p)v_i(q)|\leq ct^{-n/2}$ when $t\leq 1$.
\begin{align*}
    |\int_0^1H_{\mathbb{R}^2}(z,0,t)H^0_Y(p,p_0,t)dt|&\leq c\int_0^1 \frac{e^{-\frac{|z|^2}{4t}}}{t^{n/2+1}}dt  \\
    &\leq \frac{c}{|z|^n} \int_{|z|^2/4}^{+\infty}e^{-\tau}\tau^{n/2-1}  \\
    &\leq\frac{c}{|z|^n}(\int_{|z|^2/4}^{+\infty}e^{-\tau}\tau^{n-2}d\tau)^{\frac{1}{2}}(\int_{|z|^2/4}^{+\infty}e^{-\tau}d\tau)^{\frac{1}{2}} \\
    &\leq c\frac{e^{-|z|^2/8}}{|z|^n}
\end{align*}

Since the slowest decaying rate is of $\frac{e^{-\sqrt{\lambda_1}|z|}}{\sqrt{|z|}}$ for $|z|\to \infty$ in the above estimates, we derived that $|G^0((z,p),(0,p_0))|\leq c\frac{e^{-\sqrt{\lambda_1}|z|}}{\sqrt{|z|}}$.
\end{proof}

We also need the following estimates for higher order derivatives on Green's function.

\begin{lemma}
For $k=1,2$, when $d\to 0$, $|\nabla^kG^0((z,p),(0,p_0))|\leq c\frac{1}{d^{n-2+k}}$. When $d\to \infty$, $|\nabla^kG^0((z,p),(0,p_0))|\leq c\frac{e^{-\sqrt{\lambda_1}|z|}}{\sqrt{|z|}}$ and $d\approx |z|$.
\end{lemma}

\begin{proof}
The first statement follows from the standard argument on the Green function. 
For the second one, noticing that in the proof of the previous proposition, we only use the facts that $|\sum_ie^{-\lambda_it}v_i(p)v_i(q)|\leq ct^{-n}$ when $t\leq 1$ and $\sum_ie^{-\lambda_it}v_i(p)v_i(q)\approx ce^{-\lambda_1t}$ when $t>1$. These two are automatically satisfied for higher derivatives on the heat kernel on a compact Riemannian manifold.
    
\end{proof}

\subsection{Solving Poisson Equation on $\ALGot$ Manifolds}
In this subsection we will solve the Poisson equation $\Delta_g u=f\in \calc_1^{\infty}$ on $\ALGot$ manifolds and give the estimates on the derivatives of the solutions. Firstly, we reduce this Poisson equation to another one in which the potential $f$ has faster decay and is also integrable. 

\begin{lemma}\label{lem3.3.1}
    If $(M, J, g)$ is an $\ALGot$ Kähler manifold. Then the solvability of Poisson equation $\Delta_g u = f\in \calc^{\infty}_1$ is equivalent to the solvability of Poisson equation $\Delta_g u' = f'$, where $f'\in\calc^{\infty}_{\mu}$ for some $\mu>2$ and $\int_{M}f' = 0$. 
\end{lemma}

\begin{proof}
    We firstly prove this lemma when $M=\bbr^2\times D$ is a product $\ALGot$ manifold.
    In the following proof we will use $r$ to denote the distance function to the origin in $\bbr^2$. Since $g-g_p\in \calc^{\infty}_1$, we know that when $r$ is sufficiently large, it is comparable with the distance function to som fixed point under $g$.

    From the previous discussion we can find solution to $\Delta_{g_P} u_1 = f$, in which $u_1=v_1 + v_2$ and $v_1\in \calg^{\infty}_{r\log r}(\bbr^2)$ and $v_2\in \calc_1^{\infty}(\bbr^2\times D)$. This is because we can decompose $f$ into tow functions by integrating along fibers.
    
    We then need to compare the difference between $\Delta_g$ and $\Delta_{g_P}$. Denote $\eta = \om_g-\om_{g_P} \in \calc^{\infty}_{1}$.
    For any $u\in C^{\infty}(M)$, we have
    \begin{align}
        \Delta_g u = \frac{ni\ppb u\wedge\om_g^{n-1}}{\om_g^n}&=\frac{ni\ppb u\wedge(\om_{g_P}^{n-1} + \sum_{i=1}^n\om_{g_P}^i\wedge \eta^{n-1-i})}{\om_{g}^n} \nonumber\\
        &=\frac{ni\ppb u\wedge\om_{g_P}^{n-1}}{\om_g^n} + \frac{ni\ppb u\wedge\sum_{i=1}^n\om_{g_P}^i\wedge \eta^{n-1-i}}{\om_g^n} \\
        &=\frac{ni\ppb u\wedge\om_{g_P}^{n-1}}{\om_g^n} + O'(r^{-1})\nabla^2 u + O'(r^{-2})\nabla u \nonumber
    \end{align}
    Notice that we have an rough estimation on 
    \begin{align}
        i\ppbJ u = i\ppbJP u + d(J-J_P)du = i\ppbJP u + O'(r^{-1})\nabla^2 u + O'(r^{-1})\nabla u,
    \end{align}
    and $\frac{\om_{g_P}^n}{\om_g^n} - 1\in C^{\infty}_1$, we have
    \begin{align}\label{eq3.3.3}
        \Delta_g u- \Delta_{g_P}u &= \frac{d(J-J_P)*du\wedge\om_{g_P}^{n-1} + (d(J-J_P)*du)^2\wedge\om_{g_P}^{n-2}}{\om_g^n}\nonumber\\
        &+ O'(r^{-1})\nabla^2u + O'(r^{-2})\nabla u
    \end{align}
    From (\ref{eq3.3.3}) we derive $\Delta_g v_2- \Delta_{g_P}v_2 \in \calc^{\infty}_{2}$. 

    Notice that $v_1$ is a function on $\bbc$. We have 
    \begin{align}
        i\ppbJ v_1 &= i\ppbJP v_1 + d(J-J_P)dv_1 \nonumber\\
        &= \Delta_{g_P}v_1(\idd) + d(\frac{\partial v_1}{\partial w}K\circ dw + \frac{\partial v_1}{\partial \bar{w}}K\circ d\bar{w}) \nonumber\\
        &=\Delta_{g_P}v_1(\idd) + \nabla^2 v_1*\bar{\partial}_{J} w + \nabla v_1 * i\ppbJ w \nonumber\\
        &=\Delta_{g_P}v_1(\idd) + O(\frac{\log r}{r^2}). \label{eq3.3.4}
    \end{align}
   So for any $\eps < 1$ we have $\Delta_g v_1- \Delta_{g_P}v_1 \in\calc^{\infty}_{2-\eps}$.

    Take $u = \Tilde{u} + u_1$, then $\Delta_g u=f$ is equivalent to solving 
    \begin{align}
        \Delta_g\Tilde{u} = f - (\Delta_g-\Delta_{g_P})u_1\doteq \Tilde{f}\in \calc^{\infty}_{2-\eps}.
    \end{align}

    We can repeat the above step one more time to deduce that the initial Poisson equation is equivalent to solving $\Delta_g\Tilde{\Tilde{u}} = \Tilde{\Tilde{f}}\in \calc^{\infty}_{\mu}$ for some $\mu> 2$. So $\Tilde{\Tilde{f}}$ is now integrable.

    Take $\phi = \chi\log r$, we have
    \begin{align}
        \int_{B_R(0)\times D}\Delta_{g_P}\phi\;\text{dVol}_{g_P} = \int_{\partial(B_R(0))\times D}\partial_{\nu}\phi \;\text{dVol}_{g_P} = 2\pi
    \end{align}
    Similarly we have $\int_{B_R(0)\times D}\Delta_{g}\phi\;\text{dVol}_{g} = c>0$. Use the same method we can show $\Delta_g \phi = (\Delta_g-\Delta_{g_P})\phi + \Delta_{g_P}\phi = O'(r^{-3})$.
    By subtracting a multiple of $\Delta_{g}\phi$ from $\Tilde{\Tilde{f}}$, we can assume that $\int \Tilde{\Tilde{f}} \text{dVol}_{g}=0$.

    Now we discuss the case where $M$ is not a product. We first fix a diffeomorphism and view the potential $f$ as the restriction of some global function $f'$ on the product manifold. Then we solve $\Delta_{g_P} u'=f'$ on the product manifold. This solution gives exactly the solution on the end of $M$ under the product metric. Let $\zeta$ be a cutoff function on $M$ and vanishes in $K$. $\zeta\cdot u'$ is now a function on $M$ and solves $\Delta_{g_P} \zeta\cdot u'=f'$ when $R\gg 1$. The analysis on the new potential and some further modification follows the same routine as before.
\end{proof}

Secondly, we solve the Poisson equation on $\ALGot$ with this modified potential. The proof of the following proposition comes mainly from \cite{Hein-10}. The key difference is that we used our new Sobolev inequality.

\begin{prop}[\cite{Hein-10}]\label{ALGpoisson}
    For any $f\in \calc^{\infty}_{\mu}$ where $\mu>2$ and $\int_{M} f=0$, there exist $u\in C^{\infty}$ such that $\Delta u=f$ and $\int |\nabla u|^2<\infty$.
\end{prop}

\begin{proof}
    Suppose $\Om\subset M$ is a bounded open subset, and $u$ is the unique solution to $\Delta u = f|_{\Om} - f_{\Om}$ on $\Om$ and $u=0$ on $\partial \Om$. Here $f_{\Om}$ is the average of $f$ on $\Om$. Actually we can get $L^{\infty}$ bound of $f_{\Om}$. Suppose $B(x_0, r)\subset B(x_0, 2r)$, then $|f_{\Om}|\leq\frac{1}{|\Om|}\int_{M\Om}|f|\leq Cr^{-\mu}$. In the following argument we should think of $\Om = B(x_0, r)$ for any $r \gg 1$. And we will derive uniform estimates on $u$, which don't depend on $\Omega$.
    
    For any $p>1$, we multiply both sides of the equation by $u|u|^{p-2}$ and integrate over $\Om$ to get
    \begin{align*}
        \int_{\Om}|\nabla|u|^{\frac{p}{2}}|^2 &= \frac{p^2}{4(p-1)}(\int_{\partial\Om}u|u|^{p-2}\partial_{\nu}u -\int_{\Om}u|u|^{p-2}(f|_{\Om}-f_{\Om})) \\
        &=-\frac{p^2}{4(p-1)}\int_{\Om}u|u|^{p-2}(f|_{\Om}-f_{\Om})
    \end{align*}
    The above equation still hold true if we replace $u$ by $u+\lambda$ for any constant $\lambda$. This is because $u+\lambda$ also solves the Poisson equation, with boundary value being $\lambda$. Since $\int_{\partial\Om}\partial_{\nu}u = -\int_{\Om}\Delta u = 0$ still holds true, the boundary term in the equation will vanish.

    Using Prop and $p=2$, $\lambda = \bar{u}_{\eps}$. If $\eps \leq \mu - 2$, we have
    \begin{align*}
        ||r^{-1-\eps}(u-\bar{u}_{\eps})||_{2\sigma}^2\leq C||\nabla u^2_2||&\leq C ||r^{-\mu}(u-\bar{u}_{\eps})||_1 \\
        &\leq C||r^{-\mu + 1 + \eps}||_{(2\sigma)^*}\cdot ||r^{-1-\eps}(u-\bar{u}_{\eps})||_{2\sigma} \\
        &\leq C'||r^{-1-\eps}(u-\bar{u}_{\eps})||_{2\sigma}
    \end{align*}
    This implies that $||r^{-1-\eps}(u-\bar{u}_{\eps})||_{2\sigma}$ is uniformly bounded.

    Then we take $p=2\sigma^{k}$ and $u = u-\bar{u}_{\eps}$ to get
    \begin{align*}
        ||r^{-1-\eps}(|u-\bar{u}_{\eps}|^{\sigma^k} - \overline{|u-\bar{u}_{\eps}|^{\sigma^k}})||_{2\sigma}^2\leq C ||\nabla |u-\bar{u}_{\eps}|^{\sigma^k}||^2_2\leq C \sigma^k ||r^{-\mu}|u-\bar{u}_{\eps}|^{2\sigma^k -1}||_1
    \end{align*}
    Denote $u_k = |u-\bar{u}_{\eps}|^{\sigma^k}$. Then take $\alpha = \frac{2\sigma^k}{2\sigma^k-1}$, $\beta = 2\sigma^k$, ${l=(1+\eps)(2\sigma-\frac{1}{\sigma^{k-1}})\leq 2\eps(1+\sigma)}$ 
    \begin{align*}
        ||r^{-\mu}|u-\bar{u}_{\eps}|^{2\sigma^k -1}||_1 &\leq ||r^{l-\mu}||_{\beta} ||r^{-l}|u-\bar{u}_{\eps}|^{2\sigma^k -1} ||_{\alp}    \\
        &\leq C||r^{2\sigma(1+\eps)-\mu}||_{\beta}||r^{-1-\eps}u_k||_{2\sigma}^{(2\sigma-\frac{1}{\sigma^{k-1}})} \\
        &\leq C\max \{1,||r^{-1-\eps}u_k||_{2\sigma}^{2\sigma}\}
    \end{align*}
    The last inequality holds true if we first pick $2\sigma< \mu$ then pick $\eps$ such that $2\sigma(1+\eps)< \mu$. The constant $C$ in the above inequality doesn't depend on the region $\Omega$.
    Observed that if $\sigma$, we have
    \begin{align*}
        ||r^{-1-\eps}\bar{u}_k||_{2\sigma}^2 &= ||r^{-1-\eps}||_{2\sigma}^2||\psi_{\eps}u_k||_1^2 \\
        &\leq C ||r^{(1+\eps)\sigma-2\eps-2}||_2^2||r^{-1-\eps}u_k||_{2\sigma}^{2\sigma}\\
        &\leq C' ||r^{-1-\eps}u_k||_{2\sigma}^{2\sigma}
    \end{align*}
    The last inequality holds true if we take $\sigma < 2$. Then we need to pick $\eps\in (\frac{\sigma-1}{2-\sigma}, \frac{\mu}{2\sigma}-1)$. $\frac{\mu}{2\sigma}-1 > \frac{\sigma-1}{2-\sigma}$ is automatically guaranteed given by $\mu > 2$. Combining all the above inequalities we have 
    \begin{align*}
        ||r^{-1-\eps}|u-\bar{u}_{\eps}|^{\sigma^k}||_{2\sigma}^2\leq C\sigma^k\max \{1, ||r^{-1-\eps}|u-\bar{u}_{\eps}|^{\sigma^{k-1}}||_{2\sigma}^{2\sigma}\}
    \end{align*}

    Denote $\text{dV}' = r^{-1-\eps}\text{dV}|_{\Om}$. Then these are uniform mearsures with respect to $\Om$. The above inequality can be rewrite into
    \begin{align*}
        ||u-\bar{u}_{\eps}||_{2\sigma^k,\text{dV}'}\leq (C\sigma^k)^{\frac{1}{2\sigma^k}}\max \{1, ||u-\bar{u}_{\eps}||_{2\sigma^{k-1},\text{dV}'}\}
    \end{align*}
    Since $||u-\bar{u}_{\eps}||_{L^{\infty}}=\lim_{k\to +\infty}||u-\bar{u}_{\eps}||_{2\sigma^k,\text{dV}'}$ and $\prod_{k=2}^{+\infty}(C\sigma^k)^{\frac{1}{2\sigma^k}} < C'$, we get the uniform bound $|u-\bar{u}_{\eps}|\leq C$. Since $u=0$ on the boundary, we have $|\bar{u}_{\eps}|\leq C$, which implies that $|u|\leq C$. 

    The higher order estimates follows from the interior estimates for Poisson equations. We can take an exhaustion $\Om_i$ of $M$ together with the solutions $\Delta u_i = f|_{\Om_i}-f_{\Om_i}$ for each open region. By passing to a subsequence of these solutions they will converge to a smooth solution $u$ on the whole manifold $M$ satisfying $\Delta u = f$. When taking $p=2$ and noticing that we have uniform $L^{\infty}$ estimate on $u$, we get $\int_M|\nabla u|^2 \leq C$.
\end{proof}

To estimate the decay rate of derivatives of $u$, we first invoke a weighted Schauder estimate due to Hein and Tosatti \cite{H-T-20}.

\begin{prop}[\cite{H-T-20}]
    Let $(M, g)=(\bbr^d\times D, g_P=g_{\bbr^d}+g_D)$, where $(D,g_D)$ is a compact Riemannian manifold. Then for all $k\in \bbn_{\geq 2}$, $\alp\in (0,1)$ there is a constant $C_k=C_k(\alp)$ such that for all $p\in \bbr^d\times D$ and $0<\rho<R$,
    \begin{align}
        [\nabla^{k,g_P} u]_{C^{\alp}(B(p,\rho))} \nonumber
        \leq C_k[\nabla^{k-2,g_P}\Delta_{g_P} u]_{C^{\alp}(B(p,R))}+ (R-\rho)^{-k-\alp}||u||_{L^{\infty}(B(p,R))})
    \end{align}
    for all $u\in C^{k,\alp}_{\text{loc}}(B(p,2R))$.
\end{prop}

We need to be careful when using this proposition, since it was proved under the product metric and doesn't depend on the complex structures.

\begin{prop}\label{prop3.12}
    Under the assumption of $\ALGone$ settings, we have $|u|\leq C$, $\nabla u\in \calc_{\delta}^{\infty}$, $\nabla^2 u\in \calc_{1+\delta}^{\infty}$. $i\ppbJ u \in \calc_{1+\delta}^{\infty}$.
\end{prop}

\begin{proof}
    Firstly we can fix a diffeomorphism $\Psi: M\setminus K\to B^c(0,R)\times D$.
    Let $M_R = \Psi^{-1}(B^c(0,R)\times Y)$ and $Q_R = \int_{M_R}|\nabla u|^2$ where $R\gg 1$. Using the equation and integration by parts, we have
    \begin{align}
        \int_{r>R}|\nabla u|^2= -\int_{r>R}uf + \int_{r=R}u\partial_{\nu}u
    \end{align}
    This equality still holds after we subtract $u$ by any constant $\lambda$. We can take $\lambda$ to be the average of $u$ over $\{r=R\}$. 
    Using the fact that $g-g_P\in \calc^{\infty}_1$ and Poincare inequality we have
    \begin{align}
        Q_R\leq CR^{-(\mu-2)}+C\int_{r=R}|u-\lambda||\nabla u|&\leq CR^{-(\mu-2)}+CR\int_{r=R}|\nabla u|^2   \nonumber\\
        &=CR^{-(\mu-2)}-CR \frac{dQ_R}{dR}.
    \end{align}
    Then it is not hard to deduce that 
    \begin{align}
        Q_R\leq C_1R^{-\frac{1}{c}}+ C_2R^{-(\mu-2)}\leq CR^{-\delta}.
    \end{align}
    where $\delta = \min(\frac{1}{C}, \mu-2)$.
    
    Denote $B=B(p,1)$ for some $p$ sufficiently far from a given point in $K$, and $u_B$ be the average of $u$ on $B$. Then $||u-u_B||_{L^2(B)}\leq Cr^{-\delta}$. Since $\Delta_g(u-u_B)=f\in \calc^{\infty}_1$ on $B$ and $\Delta_g$ is uniformly elliptic with respect to $g_P$, Moser iteration shows that $|u-u_B|\leq Cr^{-\delta}$ on a slightly smaller domain. Standard Schauder estimates shows $|\nabla^k u|\leq Cr^{-\delta}$ for $k\geq 1$.

    Next
    by checking the difference between two Laplacians, we have that $|\nabla^{k,g_P}\Delta_{g_P} u| = O(r^{-1-\delta})$. 
    For any $p\in M_R$ take $0<\frac{R}{4}<\frac{R}{2}< d(p,K)$, $\alp< \min\{\mu - 2,\delta\}$ in proposition 8.11, we get 
    \begin{align}
        [\nabla^2 u]_{C^{\alp}(B(p,\frac{R}{4}))}   \leq C_k([\Delta u]_{C^{\alp}(B(p,\frac{R}{2}))}+ \frac{4^{2+\alp}}{R^{2+\alp}}||u||_{L^{\infty}(B(p,\frac{R}{2}))})
    \end{align}
    Since $||u||_{L^{\infty}}\leq C$, we have $[\nabla^2 u]_{C^{\alp}(B(p,\frac{R}{4}))}=O(R^{-1-\delta})$. When integrating from infinity we have $|\nabla^{2,g_P}u|=O(R^{-1-\delta+\alp})$.

    Taking $\nabla_{\bbc}$ to the equation $\Delta_{g_P} u = O(r^{-1-\delta})$ and repeating the above argument using the weighted Schauder estimate, we can derive that $|\nabla^l_{\bbc}\nabla_Y u|=O(r^{-l-\delta'})$ and $|\nabla^l_{\bbc}\nabla^2_Y u|=O(r^{-l-1-\delta'})$ for any $\delta' < \delta$. We can adjust $\delta$ such that $i\ppbJP u\in \calc_{1+\delta}^{\infty}$. Since $i\ppbJ = i\ppbJP + d(Kdu)$ and $K\in \calc_{1}^{\infty}$, we have $i\ppbJ u \in \calc_{1+\delta}^{\infty}$.
    
\end{proof}

It is possible that we might have a better estimate of the decay rate of the derivatives or the growth rate of $u$. But it turns out that the estimate in proposition \ref{prop3.12} is enough to prove the final theorem.

Lastly, we arrive at our main theorem in this section:

\begin{thm}\label{thm3.11}
    Let $(M, g, J)$ be an $\ALGot$ manifold. For any $f\in \calc^{\infty}_{1}$, there is a solution to the Poisson equation $\Delta_g u=f$. Moreover, when fixing the diffeomorphism of the end, we have a decomposition of $u=u_1+u_2+u_3$ satisfying
    \begin{enumerate}
        \item $u_1(w,q)=u_1(w)$, $u_1\in \cald^{\infty}_{r\log r}$;
        \item $u_2\in \calc^{\infty}_1$;
        \item $u_3$ is bounded. $\nabla u_3\in \calc^{\infty}_{\delta}$, $\nabla^2 u_3\in \calc^{\infty}_{1+\delta}$.
    \end{enumerate}
    Moreover, $i\ppbJ u \in \calc^{\infty}_1.$ 
\end{thm}

\section{The Holomorphic Splitting Theorem}

Denote $\pi:= U\rightarrow \Delta$ corresponding to $\Phi$.
Given any K\"ahler metric $\omega_0$ over $\oM$, we can find a unique $u_0\in C^{\infty}(D)$, up to a constant, such that $\omega_0+i\partial \bar{\partial}u_0|_D$ is a Calabi-Yau metric. Then we extend $u_0$ to be constant along $\Delta$ and use a cutoff function $\eta$ pulled back from $\Delta$ to make $u_0$ vanishing away from $\Delta\times D$. $\eta$ is supported in $\Delta$ and equals 1 in $\Delta(\frac{1}{2})$. If $\Delta$ is small enough and $\pi$ is holomorphic, then $\omega_0+i\partial \bar{\partial}u_0$ is positive along each fiber over $\Delta$, and may not be positive in $(\Delta\setminus\Delta(\frac{1}{2}))\times D$. The idea to compensate for this is that we can take any sufficiently positive 2 form $\alpha_1$ support in $\Delta$ and pull it back to U to compensate for the non-positive part of the Hermition form. A trial might be $\omega_1:=\omega_0+i\partial \bar{\partial}u_0+\pi_1^*\alpha_1$ for some very positive 2-form $\alp_1$ on $(\Delta\setminus\Delta(\frac{1}{2}))$. This is very close to a Kähler form. But $\pi_1^*\alpha_1$ might not be a $(1,1)$-form. We adapt two lemmas from \cite{HHN} to modify $\pi^*\alp_1$, restating them in the infinity chart. 

\begin{lemma}[\cite{HHN}]\label{lem 4.1.3}
Let $\gamma$ be a radial $2$-form with compact support on $B^c_R(0)$. $\gamma = f\cdot idzd\bar{z} = f\cdot \frac{1}{|w|^4}idwd\bar{w}$.
\begin{enumerate}
\item There exists a unique radial function $v$ on $B^c_R(0)$ such that $v \equiv 0$ near $\partial B_R(0)$ and $i\ppbJz v = \gamma$. Here, $J_0$ is the standard complex structure on $\bbc$.
\item We have $v(w) = (\int \gamma)\log |w| + \widehat v(w)$, where $\widehat v$ is radial and bounded when $|w|\gg 1$. $\widehat v=-(\int \gamma)\log |w|$ on $B^c_R(0)$.
\item On $B^c_R(0)$ we have derivative estimates $|\partial_r \widehat{v}(w)| \leq \frac{\phi(r)}{r}$, $\mathrm{supp}(\partial_r \widehat{v}(w))\subset \mathrm{supp}(\gamma)$  and $|\partial_r^2 \widehat{v}| \leq (\frac{\phi(r)}{r^2}+\max_{w\in\partial B_r(0)}|\frac{f(w)}{r^4}|)$. Here $\phi(r) = \int_{B_r(0)\setminus B_R(0)}\gamma$.
\end{enumerate}
\end{lemma}

The proof of lemma \ref{lem 4.1.3} simply follows the same argument as in the appendix. Using this lemma, we can derive the following estimate:
\begin{cor}[\cite{HHN}]\label{corollary 4.7}
Let $v(w,x):=v(w)$ be the function on $B^c_R(0)\times D$. Then $i\ppbJP v=\pi_1^* \gamma$ and $i\ppbJ v-i\ppbJP v = -(\int \gamma) \eta + \widehat\gamma$, where 
\begin{align}
\eta = (i\partial\bar{\partial}-i\ppbJP)\log |w| = \begin{cases}0 &{\textit{horizontally}},\\
O'(|w|^{-4}) & {\textit{mixed\;directions}},\\
O'(|w|^{-3}) & {\textit{vertically}};\end{cases}\\
\widehat{\gamma} =(i\partial\bar{\partial}-i\ppbJP)\widehat v = \begin{cases}0 &{\textit{horizontally}},\\
\int \gamma \cdot O'(|w|^{-3}) + \max |f|\cdot O'(|w|^{-6}) & {\textit{mixed\;directions}},\\
\int \gamma \cdot O'(|w|^{-3}) & {\textit{vertically}}.\end{cases}
\end{align}
The implied constants here are independent of $\gamma$ and, in fact, depend only on $K$.
\end{cor}

The proof of corollary \ref{corollary 4.7} follows from the formula
\begin{align*}
    i\ppbJ v-i\ppbJP v = -\frac{1}{2}d(K\circ dv) = 2\mathrm{Im}\,(dv_w\wedge\bar{\partial}w+v_wd\bar{\partial} w).
\end{align*}

Now suppose $i\ppbJz v_1 = \alp_1$ on the disc. Denote $\omega_1:=\omega_0+i\partial \bar{\partial}(u_0 + v_1)$. For $r\gg 1$, this yields a positive $(1, 1)$-form. We will use this technique several times to ensure the positivity of the $(1,1)$-form. 

Next, we can construct a complete $\ALGot$ metric on $M$. Let $\dVd = \omega_1|_{D}^{n-1}$. We can rescale $\Omega$ by some complex number such that:
\begin{align}
    \Omega\wedge\bar{\Omega}=(\frac{1}{|z|^4}+O(|z|^{-3}))idz\wedge d\bar{z}\wedge \dVd = (1+O(|w|))dw\wedge d\bar{w}\wedge \dVd
\end{align}
The K\"ahler potential for the main part should be $|w|^2$. Under the complex structure $J$ we have:
\begin{align}
    \ppb|w|^2&=\partial w\bar{\partial}w + \partial\bar{w}\bar{\partial}\bar{w}+w\ppb\bar{w}+\bar{w}\ppb w  \nonumber\\
    &=dwd\bar{w} -2i\Image (dw\partial\bar{w}) + 2i\Image (w\ppb\bar{w}) 
\end{align}
Since $\bar{\partial}w=-w^2\bar{\partial}z = -w^2\frac{i}{2}K\circ dz\in \Gamma(T^*D)$ and $\bar{\partial}w=O(r^{-2})$, it follows that $\ppb w\in\Gamma(T^*D\wedge T^*D\oplus T^*D\wedge T^*\Delta)$ and $\ppb w=O(r^{-2})$. Additionally:
\begin{align}
    \Image (dw\partial\bar{w})&=O(\frac{1}{r^2})\in \Gamma(T^*D\wedge T^*\Delta) \\
    \Image (w\ppb\bar{w})&=O(\frac{1}{r})\in \Gamma(T^*D\wedge T^*D\oplus T^*D\wedge T^*\Delta) \label{eq4.1.13}
\end{align}
Under local coordinates, these two error terms lie in class $\calc_1^{\infty}$. We can take a cutoff function and sufficiently Focusing on the neighborhood $U$ of $D$ in $\oM$, it's easy to see that the form $\om_2 = \om_1 + i\ppb (\chi |w|^2)+ i\ppb v_2 \pi_1^*\alp_2$ satisfies 
\begin{align}
    \om_2 -\om_P \in \calc_1^{\infty},\quad \omega_2^n=(1+O(|w|))idw\wedge d\bar{w}\wedge \dVd
\end{align}
Here $v_2$ is the potential satisfying $i\ppbJz v_2 = \alp_2$ for some sufficiently positive $(1,1)$-form on $B^c_R(0)$. It is added to compensate for the negative component of the Kähler form. Now, the Ricci potential is given by:
\begin{align}
    f_2:=\log(\frac{i^{n^2}\Omega\wedge\bar{\Omega}}{\omega_2^n})\in\calc_1^{\infty}
\end{align}
By definition $(M:=\oM\setminus D, \om_2, J)$ is an $\ALGot$ manifold. The volume growth of this metric is $O(r^2)$, so $\int_M(e^{f_2}-1)\omega^n$ is not necessary finite. 

\begin{rem}
    Equation \ref{eq4.1.13} is the main technical reason why we need $\calo_{3D}(D)$ to be trivial.
\end{rem}

Lastly, we refine the Kähler metric so that the Kähler potential satisfies $f\in \calc_{\mu}^{\infty}$ for some $\mu > 2$ and $\int_M(e^f - 1)\om^n = 0$.

Using theorem \ref{thm3.11}, take $R\gg 1$ and $u=u_1+u_2+u_3$ such that $\Delta_{\om_2} u = f_2$ on $M_R:=\Psi^{-1}((B^c_R(0)\times D))$. Take a cutoff function $\chi$ such that $\chi \equiv 0$ on $M_R$ and $\chi\equiv 1$ on $M_{2R}^c$. Let $\om_3 = \om_2 + i\ppb (\chi u) + i\ppb v_3$. By analyzing the growth/decay rates of $u_i$, we have $i\ppb(\chi u)|_{\Lambda^2T^*D}=O(\frac{\log r}{r})$. The term with the slowest decay rate is $u\nabla \chi\cdot d\db w|_{\Lambda^2T^*D}$. Thus, for $R\gg1$, the Kähler form remains positive in the fiber direction. Adding a sufficiently large positive form $\pi_1^*\alpha_3$ keeps $\om_3$ positive. Now, the Ricci potential 
\begin{align}
    f_4=f_2 - \log (1+\Delta_{\om_2}u + \calq_{\om_2}(i\ppb u))\in \calc_{2}^{\infty},
\end{align}
where $\calq_{\om_2}(i\ppb u) = \sum_{k=2}^{n}{n\choose k}(i\ppb u)^k\wedge\om_2^{n-k}/\om_2^n$.
Repeating this procedure one more time yields a new Kähler metric $\om_4$ such that the Ricci potential $f_4\in \calc_{\mu}^{\infty}$ for some $\mu > 2$. And $(M,\om_4,J)$ is still an $\ALGot$ manifold. Now $\int_M(e^{f_4} - 1)\,\om_4^n$ is finite.

Assume temporarily that $\int_M(e^{f_3} - 1)\,\om_4^n< 0$. On the end, let $\pi_1: M_R\to B^c_R(0)\subset \bbc$. Take $\eps > 0 $, $r_2 > r_1 \gg R+1$ and $i\ppbJz v_4 = \alp_4(r_1,r_2)$ as a positive 2-form on $B^c_R(0)$ such that it is compact supported on $B_{r_2+1}(0)\setminus B_{r_1-1}(0)$ and is equal to $idwd\bar{w}$ on $B_{r_2}(0)\setminus B_{r_1}(0)$. Let $\om = \om_4 + \eps i\ppb v_4$. By corollary (\ref{corollary 4.7}), if $\eps$ is small enough, $\om$ remains positive, and the choice of $\eps$ is independent of $\int \alp_4\approx \pi(r_2^2-r_1^2)$. Notice that $\lim_{r_2\to +\infty}\int \alp_4(r_1, r_2) = +\infty$, then there exists $r_1, r_2$ such that: 
\begin{align}
    \int_M (\frac{i^{n^2}\Om\wedge\bar{\Om}}{\om^n}-1)\om^n = \int_M(e^{f_4}-1)\om_3^n + A(\eps,r_1, r_2) = 0.
\end{align}
Here $A(\eps,r_1,r_2)$ is the term involving $\eps i\ppb v_4$.
The positive case can be handled using the same strategy. We take $\om = \om_3 - \eps i\ppb v_4$. Then there are $\eps_0>0$, $r_0 \gg R+ 1$ such that for all $\eps < \eps_0$, $r_2>r_1>r_0$, the form $\om$ is a Kähler form. We can first fix $\eps$ small enough and then let $r_2$ go to infinity to see that there are $\eps, r_1, r_2$ that make $\int_M(\frac{\om^n}{i^{n^2}\Om\wedge\bar{\Om}}-1)\om^n=0$.

This procedure to make $\int_M (e^f-1)\om^n = 0$ still work on $\ALGot(\frac{2\pi}{n})$ manifolds.

\subsection{Existence of the ALG Calabi-Yau metric and the Holomorphic Splitting Theorem}
Suppose $M$ is an $\ALGot$ manifold satisfying the required decay properties, then we can solve the Monge-Ampère equation on $M$ if the Ricci potential has the integrability condition. Here is the concrete statement from \cite{T-Y-1990} and \cite{Hein-10}:
\begin{thm}[\cite{T-Y-1990},\cite{Hein-10}]\label{thm4.10}
    Let $f\in \calc^{2,\alpha}_{\mu}(M)$ for some $\mu > 2$. If $\int_{M}(e^f-1)\om^n=0$, then there are $\alpha'\in (0, \alpha]$ and $u\in\calc^{4, \alpha}(M)$ such that
    \begin{align}
        (\om+i\ppb u)^m=e^f\om^m.
    \end{align}
    Moreover, $\int_M|\nabla u|^2\om^m<\infty$. If $f\in \calc^{k,\alpha'}_{loc}(M)$ for some $k>2$, then $u\in \calc^{k+2,\alpha'}_{loc}(M)$.
\end{thm}
The proof mainly follows that in \cite{Hein-10}. One can easily see that $(M,\om,J)$ has a $C^{k,\alp}$ quasi-atlas for $k\geq 3$. The only difference in the proof is that we need to apply our new weighted Sobolev inequality in proposition \ref{sobolev}.

\begin{rem}
    One can check that if $(M, g, J)$ is an $\ALGot(\frac{2\pi}{n})$ manifold with $n\geq 2$, then the existence of a solution to the complex Monge-Ampère equation still holds. This happens when $M$ is the blow-up of $\bbc\times D$ along the singular varieties. See \cite{Y-22}.
\end{rem}

Since we mainly focus on the case where $f\in C^{\infty}$, we have all higher order estimates on $u$. Then the metric $\om_u$ is uniform with respect to $\om$. The integration-by-part procedure also works for the Monge-Ampère equation. Analogous to proposition \ref{prop3.12}, we derive decay estimates for the solution:
\begin{prop}
    Fix the product structure of the end and take $R \gg 1$, $B_R = B^c_R(0)\times D\subset \bbc \times D$, then there exists $\delta > 0$ and a uniform constant $C$ such that, for any $x\in B_{R+2}$, 
    \begin{align}
        |u|\leq C, \quad \nabla u\in \calc_{1}^{\infty}, \quad \nabla^2 u\in \calc_{2}^{\infty}
    \end{align}
Moreover, the Kähler manifold $(M, \om_u, J)$ is an $\ALGot$ manifold.
    
\end{prop}

\begin{proof}
    Using the same energy estimate as in proposition \ref{prop3.12} we can deduce that there exists $\delta >0$ such that for all $p = (w,x)$ at the end of $M$ and $|w|\gg R$, $|u-u_B|\leq Cr^{-\delta}$. Here $B=B(p,1)$. The standard Schauder estimate shows that $|\nabla^k u|\leq Cr^{-\delta}$. 
    Similar to the proof of proposition \ref{prop3.12}, we have 
    \begin{align}
        \nabla u\in \calc_{\delta}^{\infty},\quad \nabla^2 u\in \calc_{1+\delta}^{\infty},\quad i\ppbJ u\in \calc_{1+\delta}^{\infty}.
    \end{align}
    So, the new metric remains in the class $\ALGot$.
\end{proof}

Now, suppose $(M,\om, J)$ is an $\ALGot$ manifold equipped with a Calabi-Yau metric $\om$. The product structure near infinity provides a standard holomorphic function $w$. Then extend $w$ to the whole manifold smoothly and still denote this new global function as $w$. Keep in mind that $\bar{\partial}w=O(r^{-2})$ on the end and is in $T^*D$. Then: 
\begin{align}
    \partial\bar{\partial}w=\ppbJP w - \frac{i}{2} d(J-J_P)dw\in\calc_2^{\infty}.
\end{align}
We can find a solution $w'=w'_1+w'_2$ to the equation $\Delta w' = \Delta w$ and $w_1'=O((\log r)^2)$, $|w'_2|\leq C$. Since the real and imaginary part of $w$ are of linear growth, this means that $v := w - w'$ is not identically zero. Then $v$ satisfies $\Delta_{\om}v=0$ and is a candidate of the global holomorphic function on $M$.

\begin{thm}
    The complex-valued function $v$ defined above is a global holomorphic function. Moreover, the real and imaginary part of $v$ gives the splitting of $M = \bbc \times D$.
\end{thm}
\begin{proof}
    The Bochner-Kodaira formula gives:
    \begin{align}
        \Delta_{\bar{\partial}}|\bar{\partial}v|^2 &= |\bar{\nabla}\bar{\nabla} v|^2 + |\nabla\bar{\nabla} v|^2+2\text{Re}<\partial v, \bar{\partial}\Delta_{\bar{\partial}}v> + \text{Ric}(\bar{\partial}v, \partial v) \nonumber\\ 
        &= |\bar{\nabla}\bar{\nabla} v|^2 + |\nabla\bar{\nabla} v|^2 ,
    \end{align}
    since $\Delta_{\bar{\partial}}v = \frac{1}{2}\Delta_{\om}v=0$, and $\text{Ric} = 0$.
    Take $M_r = \Phi^{-1}(B_r^c\times D)$ and $N_r = M_r^c$. Let $\nu$ to be the outer normal vector of $\partial N_r$, using $\bar{\partial}_J v = \bar{\partial}_{J_P} v + \frac{i(J-J_P)dv}{2}$ and $\nu-\partial_r=O'(\frac{1}{r})$, we have
    \begin{align*}
        |\bar{\partial}v|^2 = O'(r^{-2\delta}), \quad \nabla_{\nu}|\bar{\partial}v|^2 = O'(r^{-1-2\delta})
    \end{align*}
    Then integration by parts shows:
    \begin{align}
        \int_{N_r} |\nabla\nabla v|^2 + |\nabla\bar{\nabla} v|^2 = \int_{N_r} \Delta_{\bar{\partial}}|\bar{\partial}v|^2 = -\frac{1}{2}\int_{\partial N_r} \nabla_{\nu}|\bar{\partial}v|^2 =O(r^{-2\delta})
    \end{align}
    Taking $r\to \infty$, we have $|\nabla\nabla v|^2 = |\nabla\bar{\nabla} v|^2 = 0$. This implies that $\text{Hess } v = 0$ and $\nabla v$, $dv$ are parallel. So $\partial v$ and $\bar{\partial} v$ are also parallel. Since $|\bar{\partial}v|^2 = O(r^{-2\delta})$, then $\bar{\partial} v = 0$, which means $v$ is holomorphic.
    Similarly, we can show $|\partial v|\equiv 1$. This means that $(M, g)$ splits as $(\bbr^2\times D', g_{\bbr^2}+ g_{D'})$ along the real part and the imaginary part of $v$.  

    Now $\frac{1}{v}$ gives a local holomorphic fibration structure in a neighborhood $U_D$ of $D$ in the initial compact Kähler manifold. We take $U_D = (\frac{1}{v})^{-1}(B(0,1))$. It is easy to check that $\frac{1}{v}: U_D \to B(0,1)\subset \bbr^2$ is a proper submersion. By Ehresmann's theorem, each fiber is diffeomorphic to the central fiber $D$, which shows $D'\cong D$. We can now assume that $w=\frac{1}{v}$ is the coordinate we choose from the beginning. Since $g-g_P = O(r^{-1})$, we have $g_{D'} = g_D$. Then $\nabla^g = \nabla^{g_P}$ shows that $\nabla |J-J_p|^2 = 0$. Since $|J-J_P| = O(r^{-1})$, we have $J=J_P$.
\end{proof}

\appendix
\section{Solving Poisson Equations on $\bbr^2$}

We are going to solve $\Delta u = f$ on $\bbr^2$ when $f$ is smooth and $f=\estim{r^{-\mu}}$ for some $\mu \geq 1$.
Moreover, we assume 
\begin{align}\label{app1.1}
    r^{\mu + k}|\nabla^k f|\leq C_k.
\end{align}
Take a smooth function
\begin{align}
    \eta(r)=\begin{cases}
        0, & 0\leq r\leq 1 \\
        \in (0, 1), & 1<r<2 \\
        1, & r \geq 2
    \end{cases}.
\end{align}
We can rewrite $f = \eta f + (1-\eta)f$. Now $\eta f$ is identically zero in the unit ball $B_1(0)$ and equals $f$ outside $B_2(0)$. $(1-\eta)f$ is compactly supported on $B_2(0)$. If we can solve the Poisson equation for these two functions, then the original can be solved.

Firstly, we assume that $f$ vanishes in the unit ball. We write $f(r,\theta) = \sumover{-\infty}{+\infty}f_n(r)e^{in\theta}$ under polar coordinate via Fourier series. Assume $u(r,\theta) = \sumover{0}{\infty}u_n(r)e^{in\theta}$. Then the Poisson equation becomes
\begin{align}\label{app1.3}
    \partial_r^2 u_n + \frac{1}{r}\partial_r u_n - \frac{n^2}{r^2}u_n = f_n, \;\; n = 0, \pm 1, \pm 2\dots
\end{align}
When $n=0$ we can directly write down the solution as
\begin{align}
    u_0(r) &=\int_0^r\frac{1}{t}\int_0^t sf_0(s)ds  \\
    \partial_ru_0(r) &= \frac{1}{r}\int_0^r sf_0(s) ds \\
    \partial_r^2 u_0(r) &= f_0(r) - \frac{1}{r}\partial_r u_0
\end{align}
The integral is well-defined since $f_0\equiv 0$ on the unit ball. By checking the decay rate, we have:
\begin{align}
    u_0(r) = \begin{cases}
        \estim{\log r} & \mu > 2, \\
        \estim{(\log r)^2} & \mu = 2, \\
        \estim{r^{2-\mu}} & \mu < 2;
    \end{cases} \\
    \partial_ru_0(r) = \begin{cases}
        \estim{\frac{1}{r}} & \mu > 2, \\
        \estim{\frac{\log r}{r}} & \mu = 2, \\
        \estim{r^{1-\mu}} & \mu \in [1,2);
    \end{cases}\\
    \partial_r^2 u_0(r) = f_0(r) - \frac{1}{r}\partial_r u_0 = \begin{cases}
        \estim{\frac{1}{r^2}} & \mu > 2, \\
        \estim{\frac{\log r}{r^2}} & \mu = 2, \\
        \estim{r^{-\mu}} & \mu \in [1,2).
    \end{cases}
\end{align}

In addition, if $\mu > 2$ and $\int_0^{+\infty}sfds = 0$, then $\partial_ru_0(r) = -\frac{1}{r}\int_r^{+\infty} sf_0(s) ds = \estim{r^{1-\mu}}$, and also vanishes on the unit ball. Similarly, we have $u_0 = \estim{r^{2-\mu}}, \partial_r^2 u_0(r) = \estim{r^{-\mu}}$. 

Denote $T = \partial^2_r + \frac{1}{r}\partial_r$, we have $\partial_r(Tu_0) = T(\partial_r u_0) - \frac{1}{r^2}\partial_r u_0$. So $T(\partial_ru_0) = \partial_r f + \frac{1}{r^2}\partial_ru_0 =: \Tilde{f}$. One can check that $\partial_ru_0$ can be solved in terms of $\Tilde{f}$ exactly in the same way as before, and $\int_0^{+\infty}s\Tilde{f}ds = 0$. Using the observation above, we get:
\begin{align}
    \partial^3_ru_0(r) = \begin{cases}
        \estim{\frac{1}{r^3}} & \mu > 2, \\
        \estim{\frac{\log r}{r^3}} & \mu = 2, \\
        \estim{r^{-\mu - 1}} & \mu \in [1, 2).
    \end{cases}
\end{align}

Now we derive the following lemma:

\begin{lemma}\label{lem3.1}
    The behavior of $u_0$ depends on the decay rate of $f$. When $f$ satisfies condition (\ref{app1.1}), we have 
    \begin{enumerate}
        \item $|\nabla^k u_0|= \estim{r^{2-\mu - k}}$ if $\mu\in[1, 2)$,
        \item $|u_0|=\estim{(\log r)^2}$, $|\nabla^ku_0|=\estim{\frac{\log r}{r^k}}$ for $k\geq 1$ if $\mu = 2$,
        \item $|u_0|=\estim{\log r}$, $|\nabla^ku_0|=\estim{r^{-k}}$ for $k\geq 1$ if $\mu > 2$
    \end{enumerate}
\end{lemma}

When $|n|\geq 1$, we let $r = e^t$. By symmetry, we can assume $n>0$. The equation (\ref{app1.3}) can be rewritten into
\begin{align}
    \partial_t^2 u_n - n^2u_n = e^{2t}f_n(e^t)
\end{align}
Let $s=nt$, we have
\begin{align}
    \partial_s^2 u_n - u_n = \frac{1}{n^2}e^{\frac{2s}{n}}f
\end{align}

The solutions to these equations depend on the values of $\mu$ and $n$. 
\begin{caseof}
  \case{$\mu>1$:}{
  Since now $f = O(e^{-\frac{\mu s}{n}})$, we can safely write 
  \begin{align}
      e^{-s}(\partial_su_n+u_n) = -\int_{s}^{+\infty}\frac{1}{n^2}e^{(\frac{2}{n}-1)\lambda}f_nd\lambda,
  \end{align}
  because $\frac{2}{n}-1-\frac{\mu}{n} < 0$. Now $e^{-s}(\partial_su_n+u_n)$ is a constant when $s < 0$ and $e^{-s}(\partial_su_n+u_n)=O(e^{(\frac{2-\mu}{n}-1)s})$ when $s\to +\infty$. We can do integration one more time to get 
  \begin{align}
      u_n = e^{-s}\int_{-\infty}^se^{2\sigma}(-\int_{\sigma}^{+\infty}\frac{1}{n^2}e^{(\frac{2}{n}-1)\lambda}f_n d\lambda) d\sigma.\label{equation3.1}
  \end{align}
  Assume $|f_n(s)|\leq C_ne^{-\frac{\mu s}{n}}$, then
  \begin{align}
      |u_n(s)|&\leq C_n\frac{e^{-s}}{n^2}\int_{-\infty}^{s}e^{2\sigma}\int_{\sigma}^{+\infty}e^{(\frac{2-\mu}{n}-1) \lambda}\chi_{s>0}d\lambda d\sigma \\
      &=C_n\frac{e^{-s}}{n(n+\mu - 2)}\int_{-\infty}^{s}(e^{(\frac{2-\mu}{n}+1)\sigma}\chi_{\sigma>0} +e^{2\sigma}\cdot\chi_{\sigma\leq0})d\sigma \\
      &=C_n\left(\frac{1}{2n(n+\mu-2)}(e^s\chi_{s\leq 0}+e^{-s}\chi_{s>0}) \right.
      \\
       &\phantom{=C_n\Big(}\left.+ \frac{1}{n^2-(2-\mu)^2}(e^{\frac{2-\mu}{n}s}-e^{-s})\chi_{s>0} \right)
  \end{align}
  where $\chi_{s>0}$ is the characteristic function. So $u_n\to 0$ when $s\to 0$ and when $s\to +\infty$
  \begin{align}
    |u_n| \leq C_n'\cdot\begin{cases}
    \begin{aligned}
    \left(\frac{1}{4n^2} + \frac{1}{(2-\mu)^2-n^2}\right)e^{-s} &= \estim{r^{-n}} && \frac{2-\mu}{n}+1<0, \\
    \frac{1}{4n^2}se^{-s} &= \estim{r^{2-\mu}\log r} && \frac{2-\mu}{n}+1=0, \\
    \frac{1}{n^2-(2-\mu)^2}e^{\frac{2-\mu}{n}s} &= \estim{r^{2-\mu}} && \frac{2-\mu}{n}+1>0.
    \end{aligned}
    \end{cases}
    \end{align}
  }
  \case{$\mu = 1$ and $n > 1$:}{
  In this case, since we still have $\frac{2}{n}-1-\frac{\mu}{n}<0$, the solutions can also be derived from equation(\ref{equation3.1}). Then $|u_n|=\estim{r^{2-\mu}}$.
  }
  \case{$\mu = 1$ and $n = 1$:}{
  In this case, we use different integration formulas according to the decomposition $\partial_s^2u_n-u_n=\partial_s(\partial_s u_n - u_n) + (\partial_s u_n - u_n)$:
\begin{align}
    u_1(s) = e^{s}\int_{-\infty}^se^{-2\sigma}(\int_{-\infty}^{\sigma}e^{3\lambda}f_1 d\lambda) d\sigma. 
\end{align}
Since $|f_1|\leq C\,r^{-1}$, we have $|u_1|\leq C\, r\log r$, and $u_1$ equals $0$ when $s\leq0$.
}
\end{caseof}

Since $\partial_ru_n=\frac{n}{r}\partial_su_n$, the decay or growth rates mainly follow that $|\partial_r^ku_n|\sim \frac{n^k}{r^k}|u_n|$ when $r\gg 1$. When $\mu=1$ and $n=1$, a more delicate analysis shows that $|\partial_ru_1|=\estim{\log r}$, $|\partial^k_ru_1|=\estim{r^{-k+1}}$ for $k\geq 2$.

Define $v_n = \sumover{-n}{n}u_ie^{in\theta}-u_0$. Then $v_n$ is uniformly convergent on every compact subset in $\bbc$. Since $f$ is smooth on $\bbc$ and satisfies condition (\ref{app1.1}), we can use integration by parts to get estimates on $f_n$ involving $n$. Then we can derive higher order estimates on every compact subset. Taking the limit when $n\to \infty$ leads to a solution to $\Delta v=f-\Delta u_0$. We now summarize the behavior of the solutions $v$ according to different decaying rates $\mu$:
\begin{lemma}\label{lem3.2}
    Let $v$ be the limit of $v_n$ when $n\to +\infty$. Then we have:
    \begin{enumerate}
        \item when $\mu = 1$, $v=\estim{r\log r}$, $\partial_rv=\estim{\log r}$, $\partial^k_r v = \estim{r^{-k + 1}}$ for $k\geq 2$; 
        \item when $\mu\in (1, 2)\cup(2,3)$, $\partial^k_r v = \estim{r^{\max\{-1,2-\mu\}-k}}$ for all $k\geq 0$;
        \item when $\mu \in \{2, 3\}$, $|\partial^k_r v| = \estim{r^{2-\mu-k}\log r}$ for all $k\geq 0$;
        \item Otherwise, $\partial^k_r v = \estim{r^{-1-k}}$ for all $k\geq 0$.
    \end{enumerate}
\end{lemma}

Lastly, we still need to solve the Poisson equation when the potential $f$ has compact support.
\begin{lemma}\label{lem3.3}
    Assume $f\in C^{\infty}(\bbr^2)$ is compactly supported in $B_2(0)$, then there is a smooth solution to the equation $\Delta u=f$. Moreover, $|u|=\estim{\log r}$, $|\nabla^ku|=\estim{r^{-k}}$ for $k\geq 1$.
\end{lemma}
\begin{proof}
    Let $v$ be the solution to the Dirichlet problem:
    \begin{align}
        \begin{cases}
            \Delta v = f & \text{in  } B_3(0) \\
            v=0 & \text{on  } \partial B_3(0)
        \end{cases}.
    \end{align}
    Denote $v_1 = v(1-\eta(\frac{r}{3}))$, then $f-\Delta v_1$ is compactly supported in $B_6(0)\setminus B_3(0)$. Using Lemma(\ref{lem3.1}), (\ref{lem3.2}) we can find a solution to $\Delta v_2 = f-\Delta v_1$ satisfying $|v_2|=\estim{\log r}$, $|\nabla^kv_2|=\estim{r^{-k}}$ for $k\geq 1$. Then $u=v_1+v_2$.
\end{proof}

Combining the results in Lemma(\ref{lem3.1}), (\ref{lem3.2}) and (\ref{lem3.3}), we can derive the following estimates on the solutions to the Poisson equation:

\begin{lemma}
    Let $f$ be a smooth funtion on $\bbr^2$ satisfying the condition(\ref{app1.1}) for some $\mu\geq1$. Then there exists a solution to the Poisson equation $\Delta u = f$. And
    \begin{enumerate}
        \item if $\mu=1$, then $u=\estim{r\log r}$, $|\nabla u|=\estim{\log r}$, $|\nabla^ku|=\estim{r^{-k+1}}$ for $k\geq 2$;
        \item if $\mu\in (1, 2)$, then $|\nabla^ku|=\estim{r^{2-\mu}}$ for $k\geq 0$;
        \item if $\mu = 2$, then $u=\estim{(\log r^2)}$, $|\nabla^ku|=\estim{\frac{\log r}{r}}$ for $k\geq 1$;
        \item if $\mu> 2$, then $u=\estim{\log r}$, $|\nabla^ku|=\estim{r^{-k}}$ for $k\geq 1$.
    \end{enumerate}
    The solution is unique up to linear harmonic functions.
    
\end{lemma}

\begin{rem}
    One can also apply a weighted Schauder estimate on $\bbr^2$ to see the decay rate of higher order derivatives of $v$. By using rescaling we have 
    \begin{align}
        \sum_{l=0}^{k}R^l|\nabla^l v|_{L^{\infty}B_R(p)}+ R^{k+\alp}[\nabla^k v]_{C^{\alp}(B_R(p))} \leq C(R^{k+\alp}|\nabla^{k-2}\Delta v|_{C^{\alp}(B_{2R}(p))}+|v|_{L^{\infty}B_{2R}(p)}).
    \end{align}
    But this estimate might not be optimal for some $\mu$'s.
\end{rem}



\nocite{*}
\printbibliography

\end{document}